\documentclass[a4paper,final]{article}
\usepackage{showkeys}

\usepackage[english]{babel}
\usepackage[applemac]{inputenc}

\usepackage{amsfonts}
\usepackage{amssymb}
\usepackage{amsmath}
\usepackage[arrow, matrix, curve]{xy}
\usepackage{enumerate,xspace}
\usepackage{tikz}
\usetikzlibrary{trees}
\usetikzlibrary{snakes}
\usepackage{ellipsis}
\usepackage{ifthen}

\numberwithin{equation}{section}

\usepackage{theorem}

\usepackage{prettyref}
\newcommand{\prref}[1]{\prettyref{#1}}
\newrefformat{thm}{Theorem~\ref{#1}}
\newrefformat{lm}{Lemma~\ref{#1}}
\newrefformat{dfn}{Definition~\ref{#1}}
\newrefformat{cor}{Corollary~\ref{#1}}
\newrefformat{prop}{Proposition~\ref{#1}}
\newrefformat{sec}{Section~\ref{#1}}
\newrefformat{kap}{Chapter~\ref{#1}}
\newrefformat{ex}{Example~\ref{#1}}
\newrefformat{rem}{Remark~\ref{#1}}
\newrefformat{fig}{Figure~\ref{#1}}

\newtheorem{theorem}{Theorem}[section]

\newtheorem{lemma}[theorem]{Lemma}
\newtheorem{proposition}[theorem]{Proposition}
\newtheorem{corollary}[theorem]{Corollary}

\theoremstyle{break}

\theoremstyle{normal}
\theorembodyfont{\upshape}
\newtheorem{example}[theorem]{Example}
\newtheorem{remark}[theorem]{Remark}
\newtheorem{definition}[theorem]{Definition}

\newenvironment{proof}[1][Proof.]{
\begin{trivlist}
\item[\hskip \labelsep {\bfseries #1}]}{\hspace*{\fill}$\Box$\end{trivlist}
}

\newcommand{\ov}[1]{\overline{#1}}

\newcommand{\abs}[1]{\left|\mathinner{#1}\right|}
\newcommand{\Abs}[1]{\left\Vert\mathinner{#1}\right\Vert}
\newcommand{\floor}[1]{\left\lfloor\mathinner{#1} \right\rfloor}
\newcommand{\ceil}[1]{\left\lceil\mathinner{#1} \right\rceil}

\newcommand{\set}[2]{\left\{\, \mathinner{#1}\vphantom{#2}\: \left|\: \vphantom{#1}\mathinner{#2} \right.\,\right\}}
\newcommand{\oneset}[1]{\left\{\, \mathinner{#1} \,\right\}}
\newcommand{\smallset}[1]{\left\{\mathinner{#1}\right\}}

\newcommand{\N}{\mathbb{N}}

\newcommand{\IFF}{if and only if\xspace}
\newcommand{\homo}{homomorphism\xspace}

\newcommand{\Homos}{Homomorphisms\xspace}
\newcommand{\auto}{automorphism\xspace}

\newcommand{\morph}{morphism\xspace}
\renewcommand{\phi}{\varphi}
\newcommand{\eps}{\varepsilon}

\newcommand{\Sg}{\Sigma}
\newcommand{\Sig}{\Sigma}
\newcommand{\GG}{\Gamma}

\newcommand{\DD}{\Delta}

\newcommand{\oo}{\omega}
\newcommand{\OO}{\Omega}
\newcommand{\alp}{\alpha}
\newcommand{\Alp}{\mathop{\mathrm{alph}}}
\newcommand{\Capa}{\mathop{\mathrm{Capa}}}
\newcommand{\bet}{\beta}
\newcommand{\gam}{\gamma}
\newcommand{\del}{\delta}
\newcommand{\kap}{\kappa}
\newcommand{\lam}{\lambda}

\newcommand{\sig}{\sigma}

\newcommand{\Fab}{F^{\mathrm{ab}}}

\newcommand{\ig}{interval grammar\xspace}
\newcommand{\igs}{interval grammars\xspace}
\newcommand{\Igs}{Interval grammars\xspace}

\newcommand{\Aut}{\mathrm{Aut}}
\newcommand{\End}{\mathrm{End}}

\newcommand{\cC}{\mathcal{C}}
\newcommand{\cS}{\mathcal{S}}

\newcommand{\dist}{\mathrm{dist}}
\newcommand{\Cay}{\mathrm{Cay}}
\newcommand{\Chann}{\mathrm{Ca}}

\renewcommand{\phi}{\varphi}

\newcommand{\ggt}{\mathop{\mathrm{gcd}}}

\newcommand{\eval}{\mathrm{eval}}

 \newcommand{\leftp}{\mathrm{l}}
 \newcommand{\rightp}{\mathrm{r}}

\newcommand{\invol}{\overline{\,^{\,^{\,}}}}

\newcommand{\Oh}{\mathop{\mathcal O}}

\newcommand{\ob}{\bar{b}}
\newcommand{\oa}{\bar{a}}
\newcommand{\oc}{\bar{c}}

\newcommand{\ox}{\overline{X}}
\newcommand{\oy}{\overline{Y}}

\newcommand\lds{,\ldots ,}

\newcommand{\sse}{\subseteq}
\newcommand{\es}{\emptyset}
\newcommand{\sm}{\setminus}
\newcommand{\os}[1]{\oneset{#1}}
\newcommand{\wh}[1]{\widehat{#1}}
\newcommand{\wt}[1]{\widetilde{#1}}

\newenvironment{ok}{\noindent\color{red} olga:}{}
\newenvironment{vd}{\noindent\color{blue} volker:}{}

\newenvironment{aw}{\noindent\color{magenta} AW:}{}

\definecolor{blue}{rgb}{0.211,0.211,0.656}
\newcommand{\blue}{\color{blue}}

\begin{document}
\iftrue

\title{SLP compression for solutions of equations with constraints in free and hyperbolic groups}

\author{Volker Diekert\footnote{Universit{\"a}t Stuttgart, FMI. 
 Universit{\"a}tsstra{\ss}e 38, 
 70569 Stuttgart, Germany  \texttt{diekert@fmi.uni-stuttgart.de}}
, Olga Kharlampovich\footnote{Dept. Math and Stats, Hunter College and Graduate Center, CUNY, 695 Park Ave, New York, NY, USA, 10065 \texttt{okharlampovich@gmail.com}}\\ and Atefeh Mohajeri Moghaddam\footnote{Dept. Math and Stats, McGill University, Montreal, Canada, H3A 0B9 \texttt{mohajeri@math.mcgill.ca}}
 }
\date{August 19th, 2013}
\maketitle

\begin{abstract} The paper is a part of an ongoing program which aims to show that the existential theory in free groups (hyperbolic groups or even toral relatively hyperbolic) is NP-complete. For that we study compression of solutions with straight-line programs (SLPs) as suggested originally by Plandowski and Rytter in the context of a single word equation. We review some basic results on SLPs and give full proofs in order to keep this fundamental part of the program self-contained. 
Next we study systems of equations with constraints in free groups and more generally in free products of abelian groups. We show how to compress minimal solutions 
with extended Parikh-constraints. This type of constraints allows to express semi linear conditions as e.g.~alphabetic information. The result relies on some combinatorial analysis and has not been shown elsewhere. We show similar compression results for 
Boolean formula of equations over a torsion-free  $\delta$-hyperbolic group. The situation is much more delicate than in free groups. As  byproduct we improve the estimation of  the ``capacity'' constant used by Rips and Sela in their  paper  %\cite{rs95} 
``Canonical representatives and equations in hyperbolic groups'' from a double-exponential bound in $\delta$ to some single-exponential bound. The final section shows compression results for toral relatively hyperbolic group using the work of Dahmani: 
We  show that given a system of equations  over a fixed
toral relatively hyperbolic group, for every solution 
of length $N$ there is an SLP for another solution  such that the size of the SLP is 
bounded by some polynomial $p(s+ \log N)$ where $s$ is the size of the system.
\end{abstract}

\section*{Introduction}\label{sec:intro}
This work is motivated by the conjecture that the problem of satisfiability of a system of equations in a free group or free semigroup is NP-complete. There is a polynomial-time reduction from 
satisfiability in free groups to satisfiability in free semigroups; and 
it is also  known that this problem is NP-hard for free groups (even in the special case of quadratic equations, \cite{KharlampovichLMT10}). So one should prove that it is in NP. The roadmap how to prove this was suggested by Plandowski and Rytter in \cite{pr98icalp}. The idea is to 
prove that the length of a minimal solution is  bounded by a single-exponential function $2^{p(s)},$ where $p(s)$ is a polynomial in the size $s$ of the system of equations. Once this bound is established an NP-algorithm 
can guess a compressed version of the solution. An additional deterministic 
polynomial time algorithm can verify that the guess is indeed a solution.
The result of \cite[Thm.~3]{pr98icalp} is as follows. 
Assume that the 
length of a minimal solution of a word equation of length $s$ is bounded by some function $f(s)$. Then for a word equation of length $s$ and $f(s)$ written in binary, 
 the satisfiability of the equation 
can be decided in non-deterministic polynomial time. 
 This result was shown
 via Lempel-Ziv encodings of minimal solutions \cite[Thm.~2]{pr98icalp}, but it has been  apparent that the result 
 holds also for encodings via straight-line programs (SLPs) and for  systems of equations. Moreover it extends to Boolean formulae of equations in free groups and
 free semigroups, as shown in  \cite{hagenahdiss2000}. Actually, a more general result was established
  concerning
 systems of equations ``with rational constraints ''. 
 Rational constraints are given by regular languages (specified by NFAs, i.e., by non-determinic finite automata) which, algebraically (by the transformation monoids of the NFAs), can be reinterpreted by conditions on images in finite monoids. This approach dates back to the work of Schulz \cite{sch91} and is also explained in details e.g.{} in \cite[Ch.~12]{lot02} or \cite{dgh05IC}. Hence, for systems of equations with rational constraints
 the sizes of finite monoids become crucial.
 If the sizes of these monoid are at most  polynomial size with respect to the 
 input size of equations then \cite[Thm.~2]{pr98icalp} and \cite[Thm.~3]{pr98icalp}
 are true in this more general setting by \cite[Chapter 8]{hagenahdiss2000}.
  Monoids of polynomial size suffice to treat inequalities as constraints, but e.g.~did not allow to treat alphabetic constraints. 
  And indeed, allowing arbitrary rational constraints in the system changes the picture drastically: A similar result 
 about the existence of SLPs of size $p(s+\log f(s))$ cannot hold unless NP = PSPACE because the ``empty-intersection-problem'' for regular languages is a special case which is known to be PSPACE-complete due to a classical result of Kozen \cite{koz77}.

 In this paper we continue the research in two directions. 
 First, we deal with extended Parikh-constraints. This is slightly more general than 
 adding semi-linear constraints and strictly more general than alphabetic constraints, i.e.,
prescribing the set of letters occurring in a solution.
In the setting of extended Parikh-constraints it is very natural to extend the results 
to finitely generated free products of abelian groups.
We show that for every solution 
of length $N$ there is an SLP for another solution with the same extended Parikh-image and the same length $N$ such that the size of the SLP is logarithmic in $N$ (if $N$ is 
at least exponential in the size of the equation). 

Based on the results in the first part we show in a second part
that,  given a Boolean formula $\Phi$ of equations over a
$\delta$-hyperbolic group with generating set $\Sig$, for every solution 
of length $N$ there is an SLP for another solution such that the size of the SLP is bounded by a polynomial in $\kappa + \Abs \Phi + \log N$, where $\Abs \Phi$ is the size of the formula and 
$\kappa$ depends exponentially on $\delta$ and $\log |\Sig |$ 
(at most double-exponentially), see 
\prref{cor:boolhyp}. In the final part of the paper we consider systems of equations 
over toral relatively hyperbolic groups, and we obtain similar results with $\kappa$ 
depending exponentially on parameters of the group.
\section{Preliminaries}
\subsection{Words and monoids with involution}\label{sec:sovielumnix}
All monoids (in particular all groups) in this paper are assumed to be finitely generated. By 
$\Sigma$ (resp. $\GG^*$) we denote a finite alphabet and  
$\Sigma^*$ (resp. $\GG^*$) is the corresponding free
monoid.  (Typically, $\Sig$ is a generating set of a group $G$ and 
 $\GG= \Sig \cup \Sig^{-1}$.) Elements of free monoids are called \emph{words}.
A word in $\Sigma^*$ can be written as $w =a_1 \cdots a_n$ with $n \geq 0$ and $a_i \in \Sg$, where
$n = \abs w$ is its \emph{length}.
For $a\in \Sg$ the $a$-length of
$w$ is denoted by ${\abs w}_a$. It counts the number of $a$'s occurring in $w$.
We let
$\Alp(w) =\os{a_1 \lds a_n}$ be the \emph{alphabet} of $w$; it is the set of letters occurring in the word.  The word of length $0$ is called the
\emph{empty word}; it is denoted by $1$, since it is the neutral element of $\Sigma^*$. We have $\Alp(1) =\es$.

A \emph{factor} of a
word $w$
is a word $v$ such that $w=w_1vw_2$. A factor $v$ is a  \emph{prefix} (resp.{} \emph{suffix}), if we can write $w=vw_2$ (resp.{} $w=w_1v$).
For $0 \leq \alp \leq \bet \leq \abs w$ and $ w= a_1 \cdots a_n$ we define the factor $w[\alp, \bet]$ by
$$w[\alp, \bet] = a_{\alp+1} \cdots a_\bet.$$
Note that $\abs{w[\alp, \bet]} = \bet - \alp$. Moreover, prefixes can be written as
$w[0, \bet]$ and suffixes as $w[\alp, \abs w]$.

An \emph{involution} on a set is a bijection $\invol$ such that
$\overline{\overline{x}} = x$ for all elements $x$.
If $M$ is a monoid, then an
involution $\invol:M \to M$ must also satisfy $\overline{xy}=\overline{y}\,\overline{x}$. % for all $x,y\in M$.
If $1\in M$ is the neutral element, then $\ov 1 = 1 $
since neutral elements are unique in monoids. A \emph{\morph} between
monoids with involution is a \homo $h:M\to M'$ such that $h(\ov x) = \ov{h(x)}$. % for all $x \in M$.
%
%Note that if an NFA $A= (Q,\GG, \del, I,F)$ accepts $L \sse \GG^*$, then
%the NFA $\ov A= (Q,\GG, \ov \del, F,I)$ accepts $\ov L \sse \GG^*$, if
%$\ov \del = \set{(q,\ov a,p) }{(p,a,q)\in \del}$.
A group $G$ is viewed as a monoid with involution
by letting $\ov g = g^{-1}$ for $g \in G$. \Homos between groups are morphisms
of monoids with involution. In groups we do not distinguish between $\ov g$ and $g^{-1}$. 

If a group $G$ is generated by 
$\Sg$, then we may define  $\Gamma= \Sigma \cup \overline{\Sigma}$, where
$\overline{\Sigma}= %\set{a^{-1}\in F(\Sigma)}{a \in \Sigma} = 
\set{\ov a}{a \in \Sigma}$
is a disjoint copy of $\Sigma$. We let
$\overline{\overline{a}} =a$. This defines an involution
$\invol:\Gamma \to \Gamma$; and the involution is extended to $\GG^*$
by $\overline{a_1 \cdots a_n} = \overline{a_n} \cdots \overline{a_1}$.
Thus, $\GG^*$ is a monoid with involution, and every mapping from $\Sg$ to  another monoid $M$ with involution extends uniquely to a \morph
$\psi: \Gamma^* \to M$. Hence, $\Gamma^*$ is the free monoid with involution over $\Sg$. 
Every group element in $G$ can be represented as a word over $\GG$. 
There is a canonical \morph of  $\Gamma^*$ onto the free group $F(\Sig)$ over $\Sig$. Moreover, as a set, we identify $F(\Sig)$
with the set of \emph{reduced words}. These are the words $w \in \Gamma^*$ without any factor $a \ov a$ where $a \in \GG$. 
Reduced words are unique \emph{geodesic} normal forms for elements in $F(\Sig)$. 
For $w \in \Gamma^*$ we let $\wh w$ denote the reduced word
such that $ w = \wh w \in  F(\Sig)$.

%
%Note that every group $G$ in this paper comes with a surjective morphism 
%$\psi: \GG^* \to G$ for some finite alphabet $\GG$  with involution, but in order to increase the flexibility we may allow fixed points for the involution on $\GG$.  
%%%%%%%%%%%%%%%%%%%

\subsection{Straight-line programs}
By $\OO$ we denote a finite set of variables, which is endowed with an involution $X \mapsto \ov X$ without fixed points. 
%More precisely we let 
%$$\OO = \os{X_0,\ov{X_0}, \ldots,  X_\oo,\ov{X_\oo}.$$
% fix a linear
%order 
Hence we can write $\OO$ as a disjoint union $\OO = \OO_+ \cup \set{\ov X}{X \in \OO_+}$.
Variables occur in the context of equations and in the context of straight-line programs. For the use in straight-line programs we need to specify a partial order ${<}$ on them.

Straight-line programs are widely used, frequently in the context of algebraic circuits.
In this paper a straight-line program is a special case
of a \emph{straight-line grammar} which in turn is, by definition, 
 a reduced context-free grammar which produces exactly one word. 
 %(The precise definition of a straight-line program is given below.)
If a  grammar generates only one word, then  the grammar encodes the generated word.
In some cases the size of the generated word can be exponentially longer than the size of the grammar; and therefore straight-line grammars can be used for data compression.

\begin{example}\label{ex:aton}
Let $n \in \N$. 
\begin{enumerate}
\item 
 Consider the
following grammar with axiom $A_0$ and rules $A_{i-1} \to A_{i}A_{i}$ for $1 \leq i \leq n$ and a single terminal rule $A_{n} \to a$. The grammar has linear size in $n$,
but the axiom generates the word $a^{2^n}$ of length $2^n$.
\item (Fibonacci words) There are terminal rules $F_{1} \to b$ and $F_{2} \to a$, and for $n \geq 3$ there are rules $F_n \to F_{n-1}F_{n-2}$.
Then each $F_{n}$ generates a word $F({n})$ with length being the $n$-th Fibonacci number. Moreover, for $n \geq 3$ the word $F(n-1)$ is a prefix of $F(n)$, hence one can define an infinite sequence of letters where all $F(n)$ appear as prefixes.
\end{enumerate}
\end{example}
The following example had direct impact to algorithmic group theory. The example is due to Saul Schleimer. He used it to show that the word problem for the group $\Aut(F)$ of automorphisms of a free groups is decidable in polynomial time. We will come back to this later. 
As a matter of fact it is more convenient to consider the 
Schleimer's example in the setting of monoids. 
\begin{example}[Saul Schleimer]\label{ex:saul}
Let $M$ be a monoid generated by $\Sig$ and $A$ be a finite set of 
endomorphism of $M$; e.g., $M$ is the free group $F=F(\Sig)$ and $A$ any finite generating set for  $\Aut(F)$. Let $w= \alp_1 \cdots \alp_n \in A^*$ with $\alp_i \in A$ and $a\in \Sig$. Then the pair $(w,a)$ defines an SLP of size $\Oh(n)$ which evaluates to $\alp_1 \cdots \alp_n(a)$ as a monoid element in $M$ as follows. % (in general: of exponential size).
Variables of the SLP are denoted by $A[i,a]=  A[\alp_1 \cdots \alp_i, a ]$ for $0 \leq i \leq n$ and $a \in \Sig$. Thus, there are exactly  $\abs \Sig \cdot (n+1)$ variables. 
In order to define the rules consider first $i\geq 1$. If 
 $\alp_i( a)= b_1 \cdots b_{k}$ with $b_j \in \Sig$, then we define the production
$$A[i,a] \to A[i-1,b_1] \cdots A[i-1,b_k].$$
Finally, we define terminal rules
$$A[0,a] \to a.$$
It is clear that every variable of this grammar produces exactly one word. The variable $A[n,a]$ produces a word which yields $w(a) \in M$
with the interpretation that $w$ denotes an endomorphism of $M$ and 
$a \in M$.  
\end{example}

A straight-line program is essentially a straight-line grammar in Chomsky normal form. Formally, a
 \emph{straight-line program} (\emph{SLP} for short) is a set $S$
of rules which have either form:
\begin{align*}
X &\to a,\\
%X &\to Y[\alp,\bet],\\
X &\to YZ \text{ where } X< Y, X< \ov Y, X< Z, \text{ and } X < \ov Z
\end{align*}
Here $X \in \OO_+$, $Y,Z \in \OO$, and $a \in \GG \cup \os 1$. Moreover, we demand that  each $X\in \OO_+$ appears exactly once on a left-hand side. 

We define the \emph{height} $h(X)$ and \emph{evaluation} $\eval(X)$
for $X \in \OO$ inductively.
\begin{itemize}
\item If $X \to a$ is a rule, then  $h(X) = 1$ and  $\eval(X)= a$.
\item If $h(X)$ and  $\eval(X)$ are defined, then $h(\ov X) = h(X)$ and  $\eval(\ov X) = \ov{\eval(X)}$.
\item If $X \to YZ$ is a rule, then $h(X)=1+  \max\os{h(Y), h(Z)}$ and $\eval(X)= \eval(Y) \; \eval(Z)$.
\end{itemize}

\begin{example}\label{ex:slpforab}
Let $M$ be a commutative  monoid generated by $\Sg$. Then for each word $w \in \Sg^*$ of length $n$ there exists an SLP with $\Oh(\abs \Sg \cdot \log n)$ variables and axiom $X$ such that 
$\eval(X) = w$ in $G$. 
Indeed, every $w$ can be written in $M$ as a product 
$\gam_1^{n_1}\cdots \gam_r^{n_r}$ with $n_i \in \N$ and $r = \abs \Sg$.
\end{example}

%%%%%%%%%%%%%%%%%%%%%%%%%%%
\subsection{\Igs}\label{sec:igs}
\Igs have been introduced in the thesis of Hagenah \cite{hagenahdiss2000}.
They compress words in a very similar fashion as Lempel-Ziv compression.  The notion of \ig  is also very closely related to the notion of \emph{composition system} as defined by Gasieniec, Karpinski, Plandowski, and Rytter in \cite{GasieniecKPR96} as well as to the data structure used by Mehlhorn, Sundar, and Uhrig\cite{MehlhornSU97}.
An SLP is a special case
of a composition system, and a composition system in turn is a special case
of an \ig. Hagenah has shown how to transform an \ig into an equivalent
SLP with a quadratic blow-up in size, see \prref{thm:ig2slp}.
Thus, all three formalisms can be viewed as equivalent. As \igs provide a rather flexible formalism which is very intuitive, 
we use them here for compression.

%As above $\Omega$ is a partially ordered finite set of variables endowed with an involution without fixed. We can write $\OO = \OO_+ \cup \set{\ov X}{X \in \OO_+}$.
%By $\GG$ we denote a finite alphabet of constants with involution, which may have or may have not fixed points.

An \emph{\ig }(\emph{IG} for short) is a set $S$
of rules which have either form:
\begin{align*}
X &\to a,\\
X &\to Y[\alp,\bet],\\
X &\to Y[\alp,\bet]Z[\gam,\del]
\end{align*}
Here $X \in \OO_+$, $Y,Z \in \OO$, $\alp,\bet,\gam,\del\in \N$, and $a \in \GG \cup \os 1$. The other restrictions are listed below. 
The main idea is that if a variable $X$ evaluates to the word $w$, then $X[\alp,\bet]$ evaluates to the factor
$w[\alp,\bet]$.

In order to avoid  case distinction we treat a rule $X \to Y[\alp,\bet]$
as special case of $X \to Y[0,0]Y[\alp,\bet]$ whenever convenient.
There are several restrictions on the rules:
As for SLPs, each $X \in \Omega_+$ occurs in exactly one rule of the left hand side, and in all rules $X \to Y[\alp,\bet]Z[\gam,\del]$ we must have $X < Y$, $X < \ov Y$, $X < Z$, and $X< \ov Z$.
%Similar, in all rules $X \to Y[\alp,\bet]$ we must have $X < Y$ and $X < \ov Y$.
Next, we define the length $\abs X$ of a variable $X$ and the restrictions on  $\alp,\bet,\gam,\del$ simultaneously. If there is a rule $X \to a$, then we
let $\abs X = \abs{\ov X} = \abs a \in \os {0,1}$.
If there is a rule $X \to Y[\alp,\bet]Z[\gam,\del]$, then
$\abs X = \abs{\ov X} = \bet - \alp + \del - \gam$ and we must have
$0 \leq \alp \leq \bet \leq \abs Y$ and $0 \leq \gam \leq \del \leq \abs Z$.
In the following we assume that every \ig satisfies these restrictions.

For $w \in \GG^*$ we let $\abs w$, $h(w) = 0$, $\eval(w) =w$, and $w[\alp,\bet]$  as above. Now, we define for  $X\in \Omega\cup \GG$ and $0 \leq \alp \leq \bet \leq \abs X$ the terms $\abs X$, $h(X)$, $\eval(X)$, and $\eval([\alp,\bet])$.
The general rule is $h({\ov X})= h(X)$, $\eval({\ov X})
= \ov{\eval(X)}$, and $\eval(X[\alp,\bet])= \eval(X)[\alp,\bet]$.
Moreover, $\abs X = \abs {\eval(X)}$ and $\abs{X[\alp,\bet]}= \bet -\alp$.
Thus it is enough to define the \emph{height} $h(X)$ and \emph{evaluation} $\eval(X)$
for $X \in \OO_+$.
\begin{itemize}
\item If $X \to a$ is a rule, then  $h(X) = 1$ and $\eval(X)= a$.
\item If $X \to Y[\alp,\bet]Z[\gam,\del]$ is a rule, then $h(X)=1+  \max\os{h(Y), h(Z)}$ and $\eval(X)= \eval(Y)[\alp,\bet] \; \eval(Z)[\del,\gam]$.
\end{itemize}

For $\mu = \abs X$ and  $X[\alp,\bet]$ and we also define
$\ov{X[\alp,\bet]} % = \ov{a_{\bet}} \cdots \ov{ a_{\alp+1}}
= \ov{X}[\mu- \bet, \mu - \alp]. %%% Correct? vd
$
 In the following it is convenient to think that  for all  rules
$ X \to Y[\alp,\bet]Z[\gam,\del]$ and $X \to a$, we may also use  rules
$ \ov X \to \ov {Z[\gam,\del]}\; \ov{Y[\alp,\bet]}$ and $\ov X \to \ov a$, although our formalism does not list them explicitly.

The next proposition is used throughout in the paper. Its proof  straightforward and therefore omitted.
\begin{proposition}\label{prop:triv}
The following computation can be performed in polynomial time.
\begin{itemize}
\item Input: An \ig  $S$ and a list of words 
$w_1 \lds w_m$. % and NFAs $A_i$  for $1 \leq i \leq m$.
\item Output for each $X \in \OO$:
\begin{enumerate}
\item  The height $h(X)$ and the length $\abs X$.
\item For each $w_i$ the answer whether $w_i$ appears as a factor in $\eval(X)$.
%\item A list of triples $(i,p,q)$ such that  $\eval(X)\in L_{A_i}(p,q)$.
\end{enumerate}
\end{itemize}
\end{proposition}

%%%%%%%%%%%%%%%%%%%%%%%%%%%%%%%%%%%%%

\section{Some polynomial time algorithms}\label{sec:ns}
In this section we review some polynomial time results for certain problems involving SLPs and \igs. We survey some known results and we give full proofs. 
%We start with a result from the thesis of Hagenah. 
%The thesis was written in German and it is not published elsewhere. 

\begin{theorem}[\cite{hagenahdiss2000}, Algorithm 8.1.4]\label{thm:ig2slp}
Let $S$ be an \ig, then we can construct in polynomial time an SLP
$S'$ containing variables $X_{\alp\bet}$ for all $ X[\alp,\bet]$ which appear
in $S$ such that $\eval(X_{\alp\bet})= \eval(X[\alp,\bet])$. Moreover, we have
 $\Abs {S'} \in \Oh(\abs {\OO}^2)$.
\end{theorem}

\begin{proof}
In order to reduce the number of case distinctions we assume that there is  an $\eps$-rule $E \to 1$. (If not, we add such a rule.) Therefore we do not treat 
chain rules, because a rule $X\to Y[\alp,\bet]$ can always be  written 
as $X\to Y[\alp,\bet]E$. (Chain rules and $\eps$-rule are eliminated in a final round.)
Moreover, inside this proof it is convenient to assume that an \ig contains a rule
$X \to YZ$ \IFF it contains the  dual rule $\ov X \to \ov Z \;\ov Y$.

For every symbol  $X[\alp,\bet]$ which occurs in $S$ we define its \emph{weight}
$H(X[\alp,\bet])$ by its height $H(X[\alp,\bet]) = h(X)$ if $\alp = 0$ and 
twice its height $H(X[\alp,\bet]) =2 h(X)$ otherwise. The weight of $S$ is 
the sum of all weights. It is therefore in $\Oh(h(S) \Abs {\OO})\sse \Oh(\abs {\OO}^2)$.

We now describe a weight-reducing procedure which eliminates all symbols $X[\alp,\bet]$.
Consider a remaining $X[\alp,\bet]$ of minimal height. 
For $\bet - \alp \leq 1$ we have $\eval(X[\alp,\bet]) =a$ with $a \in \GG \cup \os 1$.
Without restriction there is a rule  $X_{a} \to a$. (If not, we add such a rule.)
We replace all occurrences of symbols $X[\alp,\bet]$ by $X_a$.
 Thus, we may assume $\bet - \alp \geq 2$.  

For $\alp >0$ we may assume that there is a rule of the form $X \to \ov YZ$, because the 
height is minimal and our assumption above.  By some simple arithmetic we  find $\gam,\del \in \N$ such that
 $$\eval(X[\alp,\bet]) =
 \eval(\ov{Y[0,\gam]}) \; \eval(Z[0,\del]).$$
 We introduce a new rule $X_{\alp\bet} \to\ov {Y[0,\gam]}Z[0,\del]$.
 After that all symbols $X[\alp,\bet]$ are replaced by the new variable $X_{\alp\bet}$.
 The height of $X_{\alp\bet}$ is $h(X)$ (but its weight is zero). 
 Since $H(Y[0,\gam]) + H(Z[0,\del]) = h(Y) + h(Z) < 2h(X)= H(X[\alp,\bet])$, this step is weight-reducing. 
 
 The remaining case is $\alp =0$.  
% For $\bet = \abs X$
% we simply replace  $X[0,\bet]$ by $X$. Hence  without restriction 
% $\alp =0$ and $2 \leq \bet < \abs X$. 
 Without restriction we have now a rule of the form 
 $ X \to YZ$.  For $\bet \leq \abs Y$ we introduce a
 new symbol $Y[0,\bet]$ and a
 rule $X_{\bet} \to Y[0,\bet]\,E$ where $E$ is the dummy symbol from above. 
 For $\bet > \abs Y$ we introduce a
 new symbol $Z[0,\gam]$ with $\gam = \bet - \abs Y$ and
 rule $X_\bet \to YZ[\gam]$. After that all symbols $X[0,\bet]$ are replaced by the
 new variable $X_\bet$. Since $H(Y[0,\bet]) = h(Y) < h(X) = H(X[0,\bet])$ and 
 $H(Z[0,\gam]) = h(Z) < h(X) = H(X[0,\bet])$, the step is again weight-reducing. 
 The number of steps and the size of the new SLP is bounded by the weight of $S$. 
  Thus, $\Abs {S'} \in \Oh(h(S) \Abs {\OO})\sse \Oh(\abs {\OO}^2)$.

 The missing transformation to deal with $\eps$- and  chain rules is standard and does not further increase the size of the SLP.
 This proves the theorem.
 \end{proof}

\subsection{Interval questions}\label{sec:iq}
The most basic question for SLPs is whether or not two variables evaluate to the same word.
It can be answered in polynomial time, thus without unfolding the word in general. This fundamental result is due to Plandowski \cite{pla94}.
His proof uses the well-known Fine-and-Wilf-Theorem. In order to keep this section self-contained we state Fine-and-Wilf and we give its proof, which is due to Jeff Shallit. 
%It can be found on his homepage, too. 
%It is the most elegant proof for Fine-and-Wilf we are aware of. 

\begin{theorem}[Fine and Wilf, 1965]\label{thm:finewilf}
Let $u,v \in \Sigma^*$ be non empty words, $s \in u\smallset{u,v}^*$ and $t \in v \smallset{u,v}^*$. Assume that  $s$ and $t$ have a common prefix of length  $\abs{u} + \abs{v} - \gcd(\abs{u},\abs{v})$, then it holds $uv = vu$. In particular, $u,v \in r^*$ where $r$ is the common prefix of $u$ and $v$ of length $\abs r = \gcd(\abs{u},\abs{v})$.
\end{theorem}

\begin{proof}
We may assume $\abs u \leq \abs v$.  
The assertion is trivial for $\abs u = 0$ or $\abs u = \abs v$. Hence we may assume $1 \leq \abs u < \abs v$.
Since $\gcd(\abs{u},\abs{v})\leq \abs{v}$, we have $v = uw$.
It remains to show $uw = wu$, because then $uv= u(uw) = u(wu) = (uw)u = vu$.
Since $\abs{s} \geq \abs{u} + \abs{v} - \gcd(\abs{u},\abs{v}) > \abs{u}$, we obtain $s \in uu\smallset{u,w}^*$. We have $t \in uw\smallset{u,w}^*$,  therefore $s' \in u\smallset{u,w}^*$ and $t' \in w\smallset{u,w}^*$ for the words $s',t'$ mit $s = u s'$ and $t = u t'$. 
Moreover $\gcd(\abs{u},\abs{v}) = \gcd(\abs{u},\abs{w})$
and $\abs v = \abs u + \abs w$, 
thus $s'$ und $t'$ have a common prefix of length $\abs{u} + \abs{w} - \gcd(\abs{u},\abs{w})$. 
By induction we conclude $uw = wu$ and hence the claim. 
The standard fact that commuting words $u$ and $v$ are powers of a common prefix is left as an easy exercise. 
\end{proof}

%In the spirit
%of \prref{sec:igs} we can solve a slightly more general question in polynomial time.%%%%%%%%%%%%%%%%
An \emph{interval question} for a given SLP is a an expression of type
\newcommand{\iq}[6]{#1[#2,#3] \overset{?}= #4[#5,#6]} %%%%interval question
$$\iq XijYk\ell.$$   %%%%$$X[i,j] \overset{?}= Y[k,\ell]$$
It is this type of interval questions is used e.g.{} in the proof of \prref{cor:maxpre}.
\newcommand{\siq}[5]{#1[#2,#3] \overset{?}= #4[#5]}   %%%% standard interval question
\newcommand{\piq}[3]{#1[#2] \overset{?}= #3[#2]}  %%%% prefix (interval) question
\newcommand{\miq}[3]{#1[#2]_{\mathrm{pf}} \overset{?}= #3[#2]_{\mathrm{sf}}}  %%%% mixed (interval) question

The expression evaluates to \emph{true}  \IFF both,
$0 \leq j -i = \ell -k \leq \min\os{\abs X, \abs Y}$  and
$\eval(X)[i,j] = \eval(Y)[k,\ell]$.

An interval question is of \emph{standard type}, if it has the form
$X[i,j] \overset{?}= Y[0,\ell]$ which we abbreviate as $\siq XijY\ell$. It is
%called a \emph{prefix question}, if it has the form $X[0,j] \overset{?}= Y[0,j]$, which we abbreviate as $\piq XjY$.
called a \emph{mixed question}, if it has the form $X[0,j] \overset{?}= Y[\abs Y - j,\abs Y]$, which we abbreviate as $\miq XjY$.
The meaning is that a prefix of length $j$ of $\eval(X)$ appears as a suffix
in $\eval(\ov Y)$. This explains the notation ``pf'' and ``sf''.
All mixed questions are of standard type.

\begin{lemma}\label{lem:finewilf}
Let $1 \leq p < q < j \leq \abs X$. Then the following three
mixed questions $\miq XjY$, $\miq X{j-p}Y$, $\miq X{j-q}Y$ evaluate to \emph{true}
\IFF the following two
mixed questions $\miq XjY$, $\miq X{j-g}Y$ evaluate to \emph{true}, where
$g=\ggt(p,q)$ is the greatest common divisor of $p$ and $q$.
\end{lemma}

\begin{proof}
 This is direct consequence of \prref{thm:finewilf}.
 \end{proof}

\begin{theorem}[\cite{AlstrupBR00,GasieniecKPR96,hagenahdiss2000,pla94}]\label{thm:pland}
The following computation can be performed in polynomial time.
\begin{itemize}
\item Input: SLP $S$ and interval questions
$\iq {X_m}{i_m}{j_m}{Y_m}{k_m}{\ell_m}$ for $1 \leq m \leq s$.
\item Output: Those questions which evaluate to \emph{true}.
\end{itemize}
\end{theorem}

\begin{proof}
 The proof follows from the next proposition.
\end{proof}

\begin{proposition}\label{prop:pland}
The following problem (involving a collection of $q$ interval questions) can be solved  in $\Oh((q + {\Abs S}^2)\cdot h(S))$ arithmetic steps.
\begin{itemize}
\item Input: SLP $S$ and interval questions
$\iq {X_p}{i_p}{j_p}{Y_p}{k_p}{\ell_p}$ for $1 \leq p \leq q$.
\item Problem: Do all  questions  evaluate to \emph{true}?
\end{itemize}
\end{proposition}

\begin{proof}
In a preprocessing phase we check that the requirements on indices are satisfied.
We also remove variables with $\abs A= 0$.
The number of arithmetic operations is in $\Oh(\Abs S)$ and can be ignored.

Now, we start the transformation process on the list of questions. In
the first phase we transform all interval questions into standard type $\siq AijB\ell$.
Consider a question $\iq AijBk\ell$ with $k\geq 1$, which is not standard. We may assume that the SLP contains a rule
$A \to CD$, because otherwise the question had standard type. Depending on the indices there are three
possibilities:
\begin{enumerate}%[1.)]
\item We can replace $\iq AijBk\ell$ by some  question
$\iq C{i}{j}Bk\ell$.
\item We can replace $\iq AijBk\ell$ by some  question
$\iq D{i'}{j'}Bk\ell$.
\item We can replace $\iq AijBk\ell$ by standard questions:
$\ov {B[k,m]} \overset{?}={\ov C}[k']$ and $\siq Bm\ell{D}{\ell'}$.
\end{enumerate}
After at most $\Oh(q \cdot h(S))$ steps we have produced a list of at most
$2q$ standard questions. Thus, without restriction, all questions 
are of standard type at the very beginning.

Next, for each pair $(A,B)$ we artificially introduce mixed questions $\miq A0B$ 
and $\miq A{\abs A}A$ (which of course evaluate to \emph{true}). 
Thus, the new number of questions is
 $Q= 2q + \abs{\OO} + \abs{\OO}^2$.
 Note that $\miq AiB$ is equivalent to
$\miq {\ov B}i{\ov A}$. Therefore, a pair of mixed questions $\miq AiB$ and $\miq AjB$
can be counted as $\miq AiB$ and $\miq {\ov B}j{\ov A}$.
In the next phases other mixed questions of type $\piq AiB$ will be generated. However, due to \prref{lem:finewilf} never more than two mixed questions need to be stored.
We now do the counting of arithmetic operations with respect to a global sum of
``euros'' which are distributed over several accounts. First,  each standard question
$\siq AijBk$ obtains an account with
$h(A) + h(B)$ euros.  The invariant is that every question on the list has always at least $h(A) + h(B)$ euros on its account. In order to do so we need initially $\Oh(Q \cdot h(S))$ euros.

Consider a standard  or mixed question $\siq Aij{B}{k}$ on our list, where the sum of heights $h(A)+ h(B)$ is maximal. If there is a rule $A \to a$ with $a \in \GG$, then we can evaluate
this question in at most $h(B)$ arithmetic operations, and then we remove it.
If the evaluation was \emph{false}, we return \emph{false} and stop.
Thus, we may assume that the SLP contains a rule
$A \to CD$. Depending on the indices there are again three
possibilities:
\begin{enumerate}[1.)]
\item We can replace $\siq Aij{B}{k}$ by some  standard question
$\siq Cij{B}{k}$.
\item We can replace $\siq Aij{B}{k}$ by some  standard question
$\siq D{i'}{j'}{B}{k}$.
\item We can replace $\siq Aij{B}{k}$ by one mixed and one standard question:
$\miq {B}\ell{C}$ and $\siq B\ell k{D}{m}$.
\end{enumerate}
Note that in all three cases the sum of heights decreased. The tricky observation is that exactly two  question of type $\miq {B}{\ell'}{C}$
and $\miq {\ov C }{\ell''}{\ov B}$ are on the list when replacing $\siq Aij{B}{k}$, because we work top-down according to the height. Thus, the only thing that happens is that some
$\miq {B}{\wt \ell}{C}$ is replaced by some $\miq {B}{m}{C}$, where
$m$ is computed according to \prref{lem:finewilf}. In all three possibilities we need only one euro to pay the of arithmetic operations, and the rest can be transferred to the new accounts without destroying the invariant\footnote{We count  a 
$\ggt$ computation on binary numbers of polynomial length as one arithmetic operation. But this  not essential because 
a more accurate amortized counting is possible. In any case it does not effect the polynomial time bound in \prref{thm:pland}.} If our list does not contain any
question anymore without that we encountered \emph{false}, then we can return
\emph{true}.
\end{proof}

\begin{remark}\label{rem:dobetter}
The time bound  in \prref{prop:pland} is not likely to be optimal. Better
time complexities might be achieved by applying recompression methods in
\cite{Jez12icalp,AlstrupBR00,MehlhornSU97}, see also \cite{Jez13dlt}.  We also refer to
\cite{Lohrey2012survey} for a recent survey on ``Algorithmics on SLP-compressed strings''.
\end{remark}

\begin{corollary}\label{cor:maxpre}
The following computation can be performed in polynomial time.
%and  in $\Oh({\Abs S}^2\cdot h(S)^2)$ arithmetic steps.
\begin{itemize}
\item Input: SLP $S$ and and variables $X,Y$.
\item Output: A number $p \in \N$ written in binary such that
the length of the longest common prefix of $\eval(X)$ and $\eval(y)$ has length $p$.
\end{itemize}
\end{corollary}

\begin{proof}
 For $X \in \OO$ we have $\abs {\eval(X)} \leq 2^{h(X)}$, hence we can solve the problem by binary search invoking at most $h(X)$ calls to \prref{thm:pland} with $s = 1$.
\end{proof}

To finish the section let us go back to the situation of a free 
group $F(\Sig)$ and
 $\Gamma= \Sigma \cup \overline{\Sigma}$.
 Recall that for $w \in \Gamma^*$ we denote by  $\wh w$ denote the uniquely defined reduced word
such that $ w = \wh w \in  F(\Sig)$.

\begin{corollary}\label{cor:freered}
The following computation can be performed in polynomial time.
%VD{and in $\Oh({\Abs S}^2\cdot h(S)^4)$ arithmetic steps. ????}
\begin{itemize}
\item Input: An SLP $S$ with constants in $\GG$. 
\item Output: An SPL $\wh S$ of size $\Oh(\Abs S \cdot h(S))$ such that
for every variable $X$ of $S$ there is a variable $\wh X$ of $\wh S$
with $\eval(\wh X) = \wh{\eval(X)}$. This means that $\wh X$ evaluates to the
reduced normal form of $\eval(X)$.
\end{itemize}
\end{corollary}

\begin{proof}
 Consider a rule $X \to YZ$. By induction on the height we may assume that we have already generated variables $\wh Y$ and $\wh Z$ such that $\eval(\wh Y) = \wh{\eval(Y)}$ and $\eval(\wh Z) = \wh{\eval(Z)}$. In addition we may assume that
  $h(Y) = h(\wh Y) $ and $h(Z) = h(\wh Z)$.
 Using \prref{cor:maxpre} we calculate the length of the
 longest common prefix of $\ov{\eval(\wh Y )}$ and $\eval(\wh Z)$. Knowing the length it
 is straightforward how to introduce new variables $Y'$ and  $Z'$ such that $\wh{\eval(X)} = \eval(Y'Z')$. For this procedure we need at most  $h(Y) + h(Z)$ 
 new rules and additional variables.
 Thus, we can introduce another variable $\wh X$ and rule $\wh X \to Y'Z'$.
 This gives us the new SLP of size $\Oh(\Abs S \cdot h(S))$.
 \end{proof}
 
 The compressed word problem can be defined in arbitrary (finitely generated) monoids $M$. For that choose some finite generating set $\Sig$. 
 The input to the \emph{compressed word problem} over $M$ is given by two  SLPs with constants in $\Sig$ and axioms $A$ and $B$ resp.
 The question is whether or not $A$ and $B$ evaluate to the same element 
 in $M$. Changing the finite set of generators does not affect whether or 
not the compressed word problem can be solved in {\bf P }or {\bf NP}.

\begin{proposition}[\cite{schleimer08}]\label{prop:sch}
Let $M$ be a finitely generated monoid and $N$ be a finitely generated submonoid of the monoid of endomorphisms  $\End(M)$. There is a polynomial-time reduction of the
word problem of $N$ to the compressed word problem of $M$. 
\end{proposition}

\begin{proof}
 The reduction is explained in \prref{ex:saul}.
\end{proof}

\begin{proposition}[\cite{lohrey06siam}]\label{prop:comploh}
Let $F$ be a finitely generated free group. Then the
compressed word problem can be solved in polynomial time.
\end{proposition}

\begin{proof}
 Compute $\wh S$ according to \prref{cor:freered} and check that $\wh X$ evaluates to $1$.
\end{proof}

\prref{prop:sch} and \prref{prop:comploh} show that the word problem of the automorphism group of
finitely generated free groups  can be decided in polynomial time \cite{schleimer08}, since 
their \auto group is finitely generated.  
More generally, the same result holds for finitely generated right-angled Artin groups, see \cite{LohreyS07} for details.
%%%%%%%%%%%%%%%%%%%%%%%%%%%%%%%%%%%%%%%%%%%
\section{Word equations with constraints}

As above we let $\OO$ be a set of variables and $\GG$ is used as an alphabet of 
constants. 
A \emph{word equation} is written as $L=R$ where $L,R \in \OO^*$, 
a \emph{constraint} is written as  $X \in \cC$ where 
$X \in \OO_+$ and $\cC\sse \GG^*$. 
A \emph{Boolean formula of equations with constraints} $\cS$ is a 
Boolean formula where the atomic propositions are either word equations 
$L=R$ or  
constraints $X \in \cC_j$.  

A \emph{solution} of $\cS$  is a \morph $\sig: \OO^* \to \GG^*$ (given by mapping  $\sig: \OO_+ \to \GG^*$) such that the Boolean formula evaluates to ``true'', if we 
substitute the atomic propositions by the corresponding truth values 
$\sig(L)= \sig(R)$ 
and $\sig(X) \in \cC$. 

A \emph{system of equations with constraints} is simply a conjunction of 
atomic propositions. Making non-deterministic guesses the existence of a 
solution of a Boolean formula can be reduced to check the existence of a 
solution for a system of equations. Since we allow constraints we may replace 
inequalities by constraints. For example, if we consider equations over a group $G$, 
an inequality $L \neq R$ can be replaced by the conjunction $L = RX \wedge X \in G \sm \os{1}$, where 
$X$ is a fresh variable. 
Thus, frequently it is enough to consider systems of equations with constraints. Moreover, we 
do not need constants. 
A constant $a \in \GG$ is replaced by a variable $A$  and the corresponding constraint $A \in \os a$. 
 %\footnote{The ``memo'' is that $g$ and $d$ refer to ``gauche'' and ``droite'';
%these are the French words of ``left'' and ``right''. The notation $d$ refers also
%to the ``denotational length'' of the equation $L=R$.}
%In order to avoid trivial cases, we assume
%$2 \leq g < d < m$. %   and  $\sigma(X_i) \neq 1$ for $1 \leq i \leq d$ whenever convenient.

%%%%%%%%%%%%%%%%%%%%%%%%%
\subsection{Free intervals}\label{sec:SectionFreeIntervals}
For the rest of the section we work with a fixed system $\cS$ and a fixed solution 
$\sig$. In \ref{sec:gen} we will define a ``generic solution'' specified by $\sig$, and we show that it can be compressed by \ig{}s. 
We write $L = X_1 \cdots X_g$ and $R = X_{g+1} \cdots X_d$ with
$X_i \in \OO$ for $1 \leq i \leq d$. Clearly, $\sig(L) = \sig (R)$. 

For a word $w \in \Gamma^*$ we call $\os{0, \ldots, |w|}$ its set of
{\em positions\/}. The idea is that letters of $w$ occur between 
positions. For positions $\alp, \bet$ %with $\alp\neq  \bet$
we call  $[\alpha, \beta]$ an \emph{interval}.
If
$0 \leq \alpha \leq \beta \leq m$, then it corresponds to the  factor $w[\alpha,\beta]$.
The  involution on  intervals is defined by   $\ov {[\alpha,\beta]} = [\beta,\alpha]$. Accordingly, we define
$w[\beta,\alpha]=\overline{w[\alpha,\beta]}$.
 An interval
 $[\alpha,\beta]$ is called  \emph{positive}, if $\alp < \bet$.

The factorization $w= \sig(X_1) \cdots \sig(X_g) = \sig(X_{g+1}) \cdots \sig(X_d)$ 
along the given solution $\sig$ ``cuts'' the word $w$ into pieces. To make this formal, we define for each $0 \leq i \leq d$ positions
$\leftp(i)$ and $\rightp(i)$ such that $\sigma(X_i)$ starts in $w$ at the left
position $\leftp(i)$ and it ends at the right position $\rightp(i)$.
Each such position is called a \emph{cut}.
Positions $0$ and $m$ are cuts and there are at most  $d$ cuts.
Clearly, if $X_i = X_j= \ov {X_k}$, then
$$w[\leftp(i), \rightp(i)] = w[\leftp(j), \rightp(j)]= w[\rightp(k),\leftp(k)].$$

Next, we are going to
define an equivalence relation $\approx$ on the set of intervals of $w$.
For that we start with a pair $(i, j)$ such that $i,j \in \os{1, \ldots, d}$ where
$X_i = X_j$ or $X_i = \overline{X_j}$. For all
$\mu, \nu \in \os{0, \ldots, \rightp(i)-\leftp(i)}$ we define a relation
between intervals $\sim$ by:

\begin{eqnarray*}
 {[}\leftp(i)+\mu, \leftp(i)+\nu{]} & \sim &
 {[}\leftp(j)+\mu, \leftp(j)+\nu{]}, \mathrm{\,if\,} X_i = X_j, \\
 {[}\leftp(i)+\mu, \leftp(i)+\nu{]} & \sim &
 {[}\rightp(j)-\mu, \rightp(j)-\nu {]}, \mathrm{\,if\,} X_i = \overline{X_j}.
\end{eqnarray*}
 By $\approx$ we denote the
reflexive and transitive closure of $\sim$. Then $\approx$ is an equivalence
relation and
$[\alpha, \beta] \approx [\alpha', \beta']$ implies both,
$[\beta, \alpha] \approx [\beta', \alpha']$ and
$w[\alpha, \beta] = w[\alpha', \beta']$. In particular, the mapping 
$[\alpha, \beta] \mapsto w[\alpha, \beta]$ {}from 
pairs of positions to $\GG^*$ is defined on equivalence classes.

\begin{definition}\label{def:freeint}
An interval
$[\alpha, \beta]$ is called  {\em free\/}, if, whenever
$[\alpha, \beta] \approx [\alpha', \beta']$,  then there is no cut
$\gamma$ with
$\min\{\alpha', \beta'\} < \gamma < \max\{\alpha', \beta'\}$.
\end{definition}

Clearly, the
set of free intervals is closed under involution, i.e.,
$[\alpha, \beta]$ is free if and only if $[\beta, \alpha]$ is free. It is also
clear that $[\alpha, \beta]$ is free if  $|\beta-\alpha| \leq 1$.

Free intervals correspond to (long) factors in the solution which are
not split to by any cut.  If the only constraints were constants, then
free intervals of length greater than $1$ could be collapsed, and therefore free intervals of length greater than $1$ 
do not appear in any solution of minimal length.
However, in order to satisfy constraints long free intervals may become important.

\begin{example}\label{ex-free}
%%%%%%%%%%LOCAL BLUBB MACRO
\newcommand{\blubb}[1]{\phantom{{}#1{}}}
Consider the following equation where  variables
$A$ and $B$ are constrained as constants by $A \in \os a$ and $B \in \os b$:

\[
   A X B \ox\,  \ov A  = Y \ov B  Y \ov A  B \oy.
\]
A possible solution in reduced words is $\sig(X) = bcb \ov c \ov b \ov b abc $, 
$\sig(Y) = abcb \ov c \ov b $ 
  \begin{align*}
    \rlap{\ensuremath{{}
    \blubb{\stackrel{0}{|}} \underbrace{\!\blubb{a \stackrel{1}{|} {}{\blue b c}
      \stackrel{3}{|}b \stackrel{4}{|}{\blue \; \oc \ob}}\!}_Y
      \blubb{\stackrel{6}{|} \ob \stackrel{7}{|}} \underbrace{\!\blubb{a \stackrel{8}{|}{\blue b
        c}{} \stackrel{10}{|} b \stackrel{11}{|} {}{\blue
        \oc \ob}}\!}_Y \blubb{\stackrel{13}{|} \oa \stackrel{14}{|} b
      \stackrel{15}{|}} \underbrace{\blubb{{\blue b c\;}\stackrel{17}{|} \ov b
      \stackrel{18}{|}{\blue \oc \ob}{} \stackrel{20}{|} \oa }}_{\ov Y}\blubb{\stackrel{21}{|}}
  {}}}
    \stackrel{0}{ |} a \stackrel{1}{|} \overbrace{{\blue b c}
      \stackrel{3}{|}b \stackrel{4}{|}{\blue \; \oc \ob}
      \stackrel{6}{|} \ob \stackrel{7}{|} a \stackrel{8}{|}{\blue b
        c}}^{X} \stackrel{10}{|} b \stackrel{11}{|} \overbrace{{\blue
        \oc \ob} \stackrel{13}{|} \oa \stackrel{14}{|} b
      \stackrel{15}{|} {\blue b c\;}\stackrel{17}{|} \ov b
      \stackrel{18}{|}{\blue \oc \ob}}^{\ox} \stackrel{20}{|} \oa \stackrel{21}{|} 
  \end{align*}
Cuts are  the eleven positions 
$0$, $1$, $6$, $7$, $10$, $11$, $13$, $14$, $15$, $20$, and $21$.
 Factors between 
 vertical bars correspond to free intervals. 
 There are three classes of free intervals, two of them result from the 
 constants $A = a$ and $B=b$. 
 There is only one equivalence class
of length
longer than 1 (up to involution), which is given by
$[20,18]\sim [1,3] \sim [8,10] \sim [20,18] \sim [13,11] \sim [15, 17] \sim [4,6]$. The solution $\sig$ says $\sig(Y)[1,3] = bc$. In principle, we can replace $\sig(Y)[1,3]$ by any other word, but the corresponding solution might be not reduced. For example, if we changed  $\sig(Y)[1,3]$ to the empty word, the resulting solution would be not reduced.  
\end{example}

\begin{definition}\label{def:mfi}
  A free interval $[\alpha, \beta]$ is called
  {\em maximal free}, if there is no free interval $[\alpha',
  \beta']$ such that both, $\alpha' \leq \min\{\alpha, \beta\}  \leq
  \max\{\alpha, \beta\}\leq \beta'$ and $|\beta-\alpha| < \beta'-\alpha'$.
\end{definition}
The following observation states  an important property of maximal free intervals.
%%%%%%%%%%%%%
\begin{proposition}[\cite{dgh05IC}] \label{prop:freecut}
Let $[\alpha, \beta]$ be a maximal free interval. Then there are intervals
$[\gamma, \delta]$ and $[\gamma', \delta']$ such that
$[\alpha, \beta] \approx [\gamma, \delta] \approx [\gamma', \delta']$ where
$\gamma$ and $\delta'$ are cuts.
\end{proposition}

\begin{proof}
By symmetry we may assume that $\alpha < \beta$. We show the existence of $[\gamma, \delta]$ where
$[\alpha, \beta] \approx [\gamma, \delta]$ and $\gamma$ is a cut. (The
existence of $[\gamma', \delta']$ where
$[\alpha, \beta] \approx [\gamma', \delta']$ and $\delta'$ is a cut follows
analogously.)

If $\alpha=0$, then $\alpha$ is a cut and we can choose
$[\alpha, \beta] = [\gamma, \delta]$. Hence let $1 \leq \alpha$ and consider the positive interval
$[\alpha-1,\beta]$. Then, for some cut $\gamma$ we have
$[\alpha-1, \beta] \approx [\alpha', \del]$ with $\min \{\alpha',
\del\} < \gamma < \max\{\alpha', \del\}$ and $|\gamma-\alpha'|=1$. A
simple reflection shows that we have
$[\alpha-1, \alpha] \approx [\alpha',\gamma]$ and
$[\alpha, \beta] \approx [\gamma, \del]$. Hence the claim.
\end{proof}

\begin{corollary}[\cite{dgh05IC}, Prop.{} 42]\label{cor:freeints}
Let $\wt{\GG}$ be the set of equivalence classes of
maximal free intervals. Then
$\wt{\GG}$ is closed under involution  and it has at most $2d-2$ elements. \end{corollary}

\begin{proof}
Let $[\alpha,\beta]$ be a maximal free interval. Then $[\beta, \alpha]$ is a maximal free interval
by definition. Hence $\wt{\GG}$ is closed under involution.
By \prref{prop:freecut} we may
assume that $\alpha$ is a cut. Say $\alpha<\beta$. Then $\alpha \neq m$
and there is no other maximal free interval $[\alpha, \beta']$ with
$\alpha<\beta'$ because of maximality. Hence there are at most $d-1$
such intervals $[\alpha, \beta]$. Symmetrically, there are at most $d-1$
maximal free intervals $[\alpha, \beta]$ where $\beta<\alpha$ and $\alpha$
is a cut.
\end{proof}
There are two types of maximal free intervals which play a quite different role. 
Those of length $1$ can be viewed as fixed of constants whereas maximal free intervals
of length greater than $1$ are specified by words which, without the presence of constraints, can be replaced be empty words in order to shorten the length of a solution.

\subsection{Generic solutions}\label{sec:gen}
In an algebraic setting the situation is now as follows. Let $X\in \OO$, we may assume that $X$ appears in the equation $L=R$. Hence
$\sig(X)$ is a factor of $w= \sig(L) = \sig(R)$. The word $w$ factorizes as a product
$w[\alpha_0,\alp_1] \cdots w[\alpha_{\ell-1},\alp_\ell]$, where
$[\alpha_0,\alp_1] \lds [\alpha_{\ell-1},\alp_\ell]$ are  maximal free intervals.
We may read this as a factorization in a word of length $\ell$ over $\wt{\GG}$.
Thus, the solution $\sig$ defines a mapping
\begin{equation}\label{eq:facsigdef}
\wt{\sig}: \OO_+ \to \wt{\GG}^*.
\end{equation}
Now, using the mapping $\oo: \wt{\GG}^* \to \GG^* $ defined above by $[\alp,\bet] \mapsto w[\alp,\bet]$ we obtain the following factorization:
\begin{equation*}\label{eq:facsig}
\sig: \OO_+ \overset{\wt{\sig}}{\longrightarrow} \wt{\GG}^* \overset{\oo}{\longrightarrow} \GG^*.
\end{equation*}
The mapping  $\wt{\sig}: \OO_+ \to \wt{\GG}^*$ is called the \emph{generic solution} specified by $\sig$. 

If $\oo': \wt \GG \to \GG^* $
is any mapping which is compatible with the involution such that
$\oo'(\wt \sig(X_j)) \in \cC_j$ for all $j$, then 
the  \morph $\oo' \circ \wt{\sig}: \OO^* \to \GG^*$ is another solution.
The  following result is closely related to  \cite{pr98icalp}.
 
 %%%%%%%%%%%%%%%%%%%%%%%%%%%%%%%
\begin{theorem}\label{thm:compall}
%There is a polynomial $p(n)$ such that the following holds.
Let  $L_i= R_i$ be a system of equations with
$L_i, R_i \in {\OO}^*$ where $1 \leq i \leq k$ and
let $\sig: \OO_+ \to \GG^*$ be  any solution. Let $d = \sum_{i=1}^k \abs{L_iR_i}$ be the denotational length,
$\wt{\sig}: \OO_+ \to \wt{\GG}^*$ the generic solution  as defined in~(\ref{eq:facsigdef}),  $\wt{N} = \abs{\wt{\sig}(L)}$ its length.

Then there is an SLP $S$ of size $\Oh(d^2 \cdot \log^2\wt{N})$  such that
each $X\in \OO_+$ appears also as variable in $S$ and satisfies $\eval(X) = \wt{\sig}(X)$.
\end{theorem}

\begin{proof}
For the purpose of the proof we may assume that $\wt\GG = \GG$ and $\wt \sig =\sig$.
We continue with the notation of above. Hence $w= \sig(L)= \sig(X_1 \cdots X_g)
= \sig(X_{g+1} \cdots X_d)$. Since $\sig$ is a generic solution, we know that all maximal free intervals have length $1$. Therefore we do not need to compress words which correspond to long free intervals.

For all cuts $\gam$ and all $\lam\in \N$ with $2^{\lam} < 2m$ we introduce a new variable $C_{\gam\lam}$ and its dual $\ov{C_{\gam\lam}}$.  The idea is that
$C_{\gam\lam}$ evaluates to the word $w[\mu, \nu]$ where
$\mu = \max\os{0,\gam - 2^{\lam}}$ and $\nu = \min\os{m,\gam + 2^{\lam}}$.

 For $\lam =0$ we have  $\mu, \nu\in \os{\gam-1,\gam,\gam +1}$.
 The interval
 $[\mu, \nu]$ corresponds to a word $u_\gam= w[\mu, \nu]\in \GG^*$ with $\abs {u_\gam} \leq 2$. In this case we introduce a rule $C_{\gam,0} \to u_\gam$.

 Now, if $\lam \geq 1$, then we begin with  an auxiliary rule
 \begin{equation}\label{otto}
C_{\gam,0} \to [\mu,\nu]C_{\gam,\lam-1}[\mu',\nu'].
\end{equation}
 Here:
 $$\begin{array}{rlllll}
\mu &= \max\os{0,\gam - 2^{\lam}}, &&
\nu &= \max\os{0,\gam - 2^{\lam-1}},\\
\mu' &= \min\os{m,\gam + 2^{\lam-1}},&&
\nu' &= \min\os{m,\gam + 2^{\lam}}.
\end{array}
$$
Without restriction we have $\mu < \nu$ and $\mu' < \nu'$.
Consider the interval $[\mu,\nu]$. There are two cases.

In the first case
$\mu -\nu = 1$. Then  $w[\mu,\nu]$ is a letter of $\GG$. In this case,  we simply substitute in (\ref{otto}) the expression $[\mu,\nu]$ by that letter. Analogously, we deal with $[\mu',\nu']$, if this is a free interval.

In the second case $\mu -\nu \geq 2$ and $[\mu,\nu]$ is not free. Then however there exists a cut $\del$
and (by symmetry and duality) $w[\mu,\nu]$ becomes the factor of some word
$\eval(C_{\del, \lam-1})[\alp, \bet]$ for suitable values $\alp, \bet$ with
$0 \leq \alp < \bet \leq 2^{\lam}$.
In this case,  we  substitute in (\ref{otto}) the expression $[\mu,\nu]$ by
$C_{\del, \lam-1}[\alp, \bet]$. Analogously, we deal with $[\mu',\nu']$.

For example, after these substitutions a rule in (\ref{otto}) might have the following form
$ %\begin{equation*}\label{gunnar}
C_{\gam,0} \to a\; C_{\gam,\lam-1}\; C_{\eta, \lam-1}[\alp', \bet'].
$ %\end{equation*}

Finally, we observe that each variable $X$ which occurs in $L=R$ is some $X_i$.
Without restriction we have $X\in \OO_+$.
For the maximal value of $\lam$ we introduce an additional chain rule
\begin{equation}\label{otto}
X \to C_{\leftp(i),\lam}[0,\abs X].
\end{equation}
  After transforming all rules in Chomsky normal form we obtain an \ig  of size $\Oh(d \cdot \log\wt N)$. The final step is the transformation of the \ig into an SLP
  using \prref{thm:ig2slp}. This establishes the bound $\Oh(d^2 \cdot \log^2\wt{N})$.
\end{proof}

According to \prref{thm:compall} we can compress the generic solution by some SLP and then we can obtain a solution in $\GG^*$ by substituting maximal free intervals. 
Say, we have a promise that a  solution exists such that $\abs{\sig(X)}$ has at most exponential length for each variable. Then we can guess in non-deterministic polynomial time an SLP for the generic solution $\wt{\sig}$. But this does not mean that we can efficiently check that $\wt{\sig}$ corresponds to an actual solution because one still has to check that there exists a substitution respecting the constraints. In order to explain the difficulty let us consider the special case of equations with rational constraints. The family of rational subsets is defined for every monoid $M$. 
 It consists of the smallest family containing the finite subsets of $M$ and which is closed under finite union, product and ``generated submonoid''. 
It has been shown in \cite{dgh05IC} that the existential theory of equations with rational constraints over free groups is PSPACE complete. The PSPACE hardness follow from the classical fact that the intersection problem for regular languages in free monoids is PSPACE complete, \cite{koz77}.
The input to that problem is simply a collection of $n$ finite (deterministic) automata
$A_1 \lds A_n$ and the question is whether  $L(A_1) \cap \cdots \cap L(A_n) \neq \es$, where $L(A_i)$ denotes the accepted language. 
(It is easy to encode this problem by a system of equations with rational constraints.) 
Now, if $L(A_1) \cap \cdots \cap L(A_n) \neq \es$ then  a shortest word in the intersection has at most exponential length. However, in general we cannot expect that there is any SLP of polynomial size representing this shortest word. If it were then we could guess the corresponding SLP in non-deterministic polynomial time and then 
check in deterministic polynomial time that the SLP generates a word in the intersection. 
As a consequence we could deduce NP=PSPACE, which is widely assumed to be false.

%%%%%%%%%%%%%%%%%%%%%%%%%%%%%%%%%%%%%%%%%%%%%%%%%%%
\section{Free products of abelian groups}
In the following $G_\alp$ denote abelian groups.
   We assume that 
each $G_\alp$ is generated by a subset $\GG_\alp\sse G_\alp\sm \os 1$ which is closed under involution, i.e., $g \in \GG_\alp$ implies  $g^{-1} \in \GG_\alp$. We let $P$ be a finite index set and 
$F= \star_{\alp \in P}G_\alp$ be the free product. Thus, $F$ is a finitely generated 
free product of abelian groups.  The direct product 
$\Fab = \prod_{\alp \in P} G_\alp$ is  the abelian 
quotient of $F$. 

We let  $\GG= \bigcup_{\alp \in P}\GG_\alp$ be the disjoint union. 
Then $\GG$ is an alphabet with involution. 
We obtain a morphism $\psi : \GG^* \to F$ and elements of $F$ can be  represented as 
words over $\GG$. 
Words in $\GG^*$ are split into factors according to $\alp$. To make this formal
we let $\DD_\alp = G_\alp \sm \os 1$ and $\DD= \bigcup_{\alp \in P}\DD_\alp$ be the disjoint union. Then $\DD$ is also an alphabet with involution, but typically infinite. The inclusions $\GG_\alp \sse \DD_\alp \sse G_\alp$ induce canonical 
\morph{}s $$\GG^* \sse \DD^* \overset{\psi}\to F \to  \Fab.$$
We also have a \morph $\psi_\alp: \DD^* \to G_\alp$
wich is induced by $\psi_\alp(g ) = g$ for $g \in G_\alp$ and 
$\psi_\alp(g ) = 1$ otherwise. 

A word $a_1 \cdots a_n$ with $a_i\in \DD$ is called \emph{reduced}, if
$a_{i} \in \DD_\alp$ implies $a_{i+1} \notin \DD_\alp$ for all $\alp \in P$ and 
$1 \leq i < n$. Every element in $F$ has a unique normal form $f_\DD$ as a reduced word 
over $\DD$. We identify the set $F$ with its set of normal forms 
$\wh F= \set{f_\DD}{f \in F} \sse \DD^*$. 
For $f \in F$ we let $\abs{f}_\DD = \abs{f_\DD}$ be the length a reduced word 
in $\DD^*$ representing $f$, whereas $\abs{f}_\GG$ denotes the length of a shortest word over $\GG^*$ representing $f$. Note that $\abs{f}_\DD \leq \abs{f}_\GG$. 
A word $w\in \GG^*$ of length $\abs{f}_\GG$ representing $f$ is called a 
\emph{geodesic} word for $f$. In contrast to $f_\DD$ geodesics are not unique, in general.

\subsection{Extended Parikh-constraints}\label{sec:epc}
For a word $w \in \DD^*$ we let $\abs{w}_\alp$ the number of letters from $\DD_\alp$. 
The vector $(\abs{w}_\alp)_\alp \in \N^P$ is called the \emph{Parikh-image} of $w$. It counts how often a position $\alp \in P$ is used as a
non trivial factor in a word over $\DD$. 
We have $\abs w = \sum_\alp \abs{w}_\alp.$
We also let  $\pi_\alp(w) = (\abs{w}_\alp, \psi_\alp(w)) \in \N \times G_\alp$. 
This extends to a unique \morph
$$ \DD^* \to \prod_{\alp \in P} (\N \times G_\alp) = \N^P \times \Fab. $$ 
However, later in the applications we need also to control the first and last positions from 
$P$, because we need that if we replace a factor by some other factor in a reduced word, the new word must be still reduced. 
Therefore we use two more mappings. 
We define $\text{first}(w)\in P \cup \os 1$ to be $1$ is empty and 
to be $\alp \in P$, if the reduced form of $w$ starts with a
non empty factor in $P$. Symmetrically, we let  $\text{last}(w)$ to be the last position. Thus, $\text{last}(w) = \text{first}(\ov w)$.

This yields %the Parikh image $(\abs{f}_\alp)_{\alp \in P} \in \N^P$ and 
an 
``extended Parikh-mapping''
\begin{align*}
\pi: \DD^* \to &\quad \N^P \times \prod_{\alp \in P} G_\alp \times (P \cup \os 1) \times (P \cup \os 1)\\
\pi(w) = &\; ((\abs{w}_\alp)_{\alp \in P}, \; \phi(w),\; \text{first}(w),\;  \text{last}(w)).
\end{align*}
Using $F = \wh F \sse \DD^*$ we can apply $\pi$ to elements in the group $F$. 
The idea is to change solutions in such a way that they become compressible by SLPs, but  the image under $\pi$ remains invariant.

For simplicity of notation we choose for every index $\alp \in P$ some fixed  letter, 
called $\alp \in \GG_\alp$ again. Thus, we view $P \sse \GG \sse \DD$ and we can speak about reduced words in $P^*$. Such a word is a sequence 
$\alp_1 \cdots \alp_n$ with $\alp_i \neq \alp_{i+1}$ for all $1 \leq i <n$. 
We have the following combinatorial lemma which is crucial for compression.

\begin{proposition}\label{prop:reorder}
Let $w= \alp_1 \cdots \alp_n \in P^*$ be a reduced sequence of length $n \geq 1$ 
with $a =  \alp_1$, $c = \alp_n$; and let $\abs {\Alp(w)} = \ell$. 

If $\ell \leq 2$, then 
$w$ has either the form $(ac)^{k}$ or $(ab)^{k}a$ for $k = \floor {\frac n2}$. 
If $\ell \geq 3$, then there exists a reduced word $w'\in P^*$ with $\pi(w)= \pi(w')$
and $w' \in aP^*c$ such that one of the following assertions hold. 
\begin{enumerate}
\item \label{reorderi}
It is  $\abs{w}_a = \ceil{\frac n2}$ and for $k = \ell -1$ and some $d \in\os { 1, a}$ we have
$$w' =  (a\bet_1)^{n_1} \cdots (a\bet_k)^{n_k}
d.$$
\item \label{reorderii}
If $\abs{w}_a = \frac n2$ %, $a = f$,  
and for $k = \ell -1$ we have 
$$w' = (a\bet_1)^{n_1} (\bet_2 a )^{n_2} \cdots (\bet_ka)^{n_k}.$$
\item \label{reorderiii}
It is  $\abs{w}_a < \frac n2$ and  for some $d \in \os{1,c}$ and  $k \leq \binom \ell 2$ we have
$$w' =(a\gam_1)^{n_1}(\bet_2\gam_2)^{n_2} \cdots (\bet_k\gam_k)^{n_k}d.$$
\end{enumerate}
\end{proposition}
\begin{proof} The proof is obvious for $\ell \leq 2$. Hence let $\ell \geq 3$. 
Note that the image of $\psi (P^*)$  lies in the abelian group $\Fab$. 
Thus, we it is enough to show  that there exists a reduced sequence $w'\in aP^*c$ with $\abs{w}_\alp = \abs{w'}_\alp$ for all $\alp \in \Alp(w)$. 
Since $w$ is reduced we have $\abs{w}_a \leq  \ceil{\frac n2}$. For
 $\abs{w}_a \geq  \frac n2$ we are in situation \ref{reorderi} or \ref{reorderii}.
 If $n$ is odd we are in situation \ref{reorderi} with $a = c = d$. 
 If $n$ is even and $a\neq c$ we are in situation \ref{reorderi} with $\bet_k = c$ and $d=1$. If $n$ is even and $a = c$ we are in situation \ref{reorderii}.

For the rest of the proof we my therefore assume $\abs{w}_a < \frac n2$. 
If $n$ is odd, then the assertion holds for the 
word $\wt w= \alp_1 \cdots \alp_{n-1} \in P^*$ with first letter 
$a$ and last letter $\alp_{n-1}$ by induction. We are done in this case with $d=c$ since $\alp_{n-1}$ exists and $\gam_k= \alp_{n-1}\neq \alp_n =c$.

It remains to show
\ref{reorderiii} under the assumption that $\abs w$ is even and $\abs{w}_a < \frac n2$.
Note that this implies $n \geq 4$. 
We match indices $1 \lds n$ by defining sets 
$V_{ij}= \os {i,j}$ such that $\alp_{i}\neq \alp_j$ and in such a way that the collection of the sets 
$V_{ij}$ yields a partition of $V= \os{1 \lds n}$. To see that this is possible start with any partition of $V$ into two-element subsets $V_{ij}$. Assume there is 
some $V_{ij}$ with $\alp_{i}=\alp_j$. As $w$ is reduced, we have $\abs{w}_{\alp_{i}}
\leq n/2.$ Hence there must be some $V_{p q }$ with $\alp_p \neq \alp_{i}\neq \alp_q$. We replace $V_{ij}$, $V_{p q}$ by $V_{ip}$, $V_{jq}$. Continuing this way we achieve a partition as desired. 

Consider the set $V_{i_nn}$ and let  $\alp_{i_n}= b$.
Then we have $b \neq c$.  We are going to construct 
a reduced word  of the form
$$w' =(a\gam_1)^{n_1} (\bet_2\gam_2)^{n_2}\cdots (\bet_{k-1}\gam_{k-1})^{n_{k-1}}(bc)^{n_{k}}.$$
 We construct $w'$ under the restriction that 
$w'$ is reduced, it begins with $a$ and it ends in the factor $bc$. 
The idea is to write 
the sets $V_{ij}$ in a list starting with some $V_{1j_{1}}$ and ending in $V_{i_nn}$; and then to replace $V_{ij}$ by $\alp_{i}\alp_j$.
We must show that the resulting word $w'$ is reduced. 
We know that $V_{ij} \neq V_{p q} $
implies $\os{\alp_{i},\alp_j} \neq \os{\alp_{p},\alp_q}$. This shows 
$k\leq \binom \ell 2$.
As $\alp_{i}\neq \alp_j$ for all $V_{ij}$ we have always two options how to continue until the last $V_{i_nn}$.
For $a=b$ we can avoid $\gam_{k-1} = b$ and therefore the construction 
of $w'$ is straightforward. Now for $a \neq b$ we have 
$a \neq b \neq c$. Hence $\abs{w}_b = n/2$ cannot happen, because $w$ is reduced of even length. 
This means there is at least one set $V_{ij}$ with $\alp_{i}\neq b \neq \alp_j$.
We may assume that the replacement of $V_{ij}$ by $\alp_{i}\alp_j$
results in a factor $\bet_m\gam_m$ for some $m \leq k-1$ such that 
we can avoid $\gam_{q} = b$ for all $m <q \leq k-1$. \end{proof}

\subsection{Equations over free products of abelian groups}\label{sec:eqfpag}
As above we continue with a free product of abelian groups $F$.
An \emph{equation} over $F$ is written as $L=R$ where $L,R \in \OO^*$. 
We do not need constants, because we allow extended Parikh-constraints.

 We use the following well-known fact. It shows that solvability of an equation over $F$  split into two parts. A global word equation over $\DD$ and local equations over the $G_\alp$. 
 Its proof is straightforward and omitted.
\begin{lemma} \label{lem:lem1}
  Let $u,v,w \in \DD^*$ be  reduced words.  Then we have
  $uv=w$ in the free product $F$
  if and only if there are $\alp \in P$, $a,b,c \in G_\alp$, $p, q, r \in \DD^*$
  such that  
  \begin{enumerate}
\item $u=pa\overline{q}$,
  $v=qbr$, and $w=p c r$ in $\DD^*$,
  \item
$ab = c$ in the abelian group $G_\alp$.
\end{enumerate}
\end{lemma}

 We construct a new system of equations $\cS'$ such that
 $\sig$ solves $\cS'$ and such that all solutions of  $\cS'$ are also solutions of 
 $\cS$. The construction is  as follows. 
  First, we transform all equations into triangular form, i.e., they look 
 like $XY=Z$ where $X,Y,Z \in \OO$. 
 
 Next, we split the triangular system of equations into two parts, a global part of word equations with solutions in $\DD^*$ and a  local part of equations of type
$AB=C$ with solutions in $G_\alp$.
Now the trick is to put $ab = c$ into constraints. More concretely consider 
an equation $XY=Z$ of our system. Let $u = \sig(X), v= \sig(Y),w = \sig(Z)\in \DD^*$ be the  reduced words given by $\sig$. 
We choose $\alp \in P$, $a,b,c \in G_\alp$, $p, q, r \in \DD^*$ according to 
\prref{lem:lem1}. We introduce fresh symbols $A,B,C,P, Q, R$  and we add them to $\Omega_+$.

In the next step we replace the equation 
$XY= Z $ in $\cS$ by three equations: 
 $$U=PA\overline{Q}, \quad 
  V= QBR, \quad W=PCR.$$  
We simulate the equation $AB=C$ by constraints. 
To do so, we introduce three additional extended Parikh-constraints: 
$$A = \os {a} , \quad B = \os {b} , \quad C = \os {c}.$$
We also extend the solution by defining $\sig(A) = a$,  $\sig(B) = b$, \ldots , $\sig(R) = r$. Moreover, we add the constraint $X\in \wh F$ for all $X\in \Omega_+$
where  $\wh F$ is the set of reduced words in $\DD^*$. 

This step finishes the transformation and defines a system $\cS'$ with constraints. 
All equations are triangular and the constraints are conjunctions of extended Parikh-constraints and constraints of the form $X\in \wh F$. 
This is not an extended Parikh-constraint!

Note that $\sig$ still solves the new system with constraints and if $\sig'$
is any other solution of the new system, then $\sig'$ solves the original system $\cS$ as well because 
$ab = c$ in the abelian group $G_\alp$. 

Finally, for a solution $\sig: \OO_+\to \DD^*$ and $X \in \Omega$ we let 
$\sig_\alp(X) \in G_\alp$ be the image of $\sig(X)$ under the natural projection 
of $\DD^*$ onto $G_\alp$. Each $\sig_\alp(X)$ can be written as a word over $\GG_\alp$, and a word  $\sig(X)$ can be written as a word over $\GG$ of 
length $\abs{\sig(X)}_\GG$.

If $\Phi$ is  a Boolean formula of equations  over $F$ with constraints and $\sig : \OO_+\to \DD^*$ is a solution then 
there we can extract a system of equations $\cS$ and subset of 
constraints (and  their negations) and additional constraints of type $X \neq 1$ 
such that $\sig$ is a solution of $\cS$ and moreover, every 
solution of $\cS$ solves $\phi$, too. 
We define the size of the formula $\Phi$ by 
 $\Abs \Phi=  \abs \GG + \sum_{i=1}^k \abs{L_iR_i}$, where 
$L_i=R_i$ are the equations used in the formula with  $L_i$, $R_i \in \OO^*$. Note that the size $\Abs \Phi$ does not take the constraints into account. For the length of a solution we have to write 
words over $DD$ as words over $\GG$. Therefore we define the 
length of $\sig$ by the number $N = \abs \OO +\sum_{X \in \OO}\abs{\sig(X)}_\GG$. The term $\abs \OO$ takes care to write down $\sig(X) =1$. 

\begin{theorem}\label{thm:gencomp}
There exists a polynomial $p(n)$ such that the following holds: 
Let $F$ be a free product of abelian groups and $\GG$ be a set of generators of $F$. 
Let $\Phi$ be a Boolean formula of equations  over $F$ with extended Parikh-constraints and
  and let $\sig : \OO_+\to \DD^*$ be a solution in reduced words
of length $N$. 

Then there is also a solution $\sig': \OO_+\to \DD^*$ in reduced words such that 
the following conditions hold.
\begin{enumerate}
\item We have $\pi(\sig(X)) = \pi(\sig'(X))$ for all $X\in \OO$.
\item There is an SLP  $S$ with constants in $\GG$ of size at most $p(\Abs\Phi + {\log N})$  such that
each $X\in \OO$ appears also as variable in $S$ and satisfies 
$\eval(X) = \sig'(X)$ in the group $F$.
 \end{enumerate}
 \end{theorem}

\begin{proof}
 With the help of Parikh constraints we may assume that the input $\Phi$ is given by some system $\cS$ of equations with extended Parikh-constraints. Next we may assume that all equations are in triangular form. According to \prref{lem:lem1} we transform the input into a system of word equations with  extended Parikh-constraints. 
 We do not need the constraints $X \in \wh F$ because $\sig$ is a solution in reduced words. 
Hence,  $\sig$ can be extended to a solution in reduced words of the new system and 
 the total length $N$ is still polynomial in its original value. 
 
 In the next step, we produce the generic solution $\wt \sig$ (which belongs to 
 $\sig$) 
 according to \prref{thm:compall}. This solution has an SLP of size which is polynomial in 
 $\abs{\Phi}+ {\log N}$. In the generic solution we must substitute  maximal free intervals by compressible words which respects the extended Parikh-constraints. Note that it is here that we need the control on first and last letters, because after substitution the words have to remain reduced. 
 In order to produce short SLPs for the substitution we use \prref{prop:reorder}.
\end{proof}

Note that \prref{thm:gencomp} does not say that every solution can be compressed using an SLP. Even if $N$ is minimal we only state that there is another solution with good compression. 
However if there are no constraints at all or, more general, if we content ourselves to ``alphabetic'' constraints, then we can state a stronger statement. 

Let $F= \prod_{\alp \in P} G_\alp$ be a free product as above. 
For an element $w \in F$ we let $\Alp(w) = \set{\alp \in P}{\abs{w}_\alp \geq 1}$ be the 
\emph{alphabet} of $w$. The alphabet specifies which factors in the free product 
are used in a reduced representation of $w$. This allows to define an 
\emph{alphabetic constraint} by 
$$\set{w \in F}{ \Alp(w)=A, \;\text{first}(w) = \bet,\;  \text{last}(w)=\gam
},$$ where 
$\alp\sse P$, $\bet, \gam \in P$. Clearly, an alphabetic constraint is just a special case of an extended Parikh-constraint. There are 
$\abs{P}^2 2^{\abs P}$ alphabetic constraints, but in formulae it is enough to have 
atomic constraints of the form $\set{w \in F}{\alp \in \Alp(w), \text{first}(w) = \bet, \text{last}(w)=\gam },$ where $\alp, \bet, \gam \in P$. 

\begin{theorem}\label{thm:genalp}
There exists a polynomial $p(n)$ such that the following holds: 
Let $F$ be a free product of abelian groups and $\GG$ be a set of generators of $F$. 
Let $\Phi$ be a Boolean formula of equations  over $F$ with  alphabetic constraints and let $\sig : \OO_+\to \DD^*$ be a solution
such that its length $N$ is minimal among all solutions. 
Then there is an SLP  $S$ of size $p(\Abs{\Phi}  + {\log N})$  such that
each $X\in \OO$ appears also as a variable in $S$ and satisfies $\eval(X) = \sig(X)$ in the group $F$.
  \end{theorem}

\begin{proof}
 The proof is almost identical to the proof of \prref{thm:gencomp}. The difference is that in order to substitute maximal free intervals of the generic solution by words we can use the words from the original solution given by $\sig$. 
 These words are necessarily the shortest ones which respect the alphabetic constraints. 
 Thus they visit at most $\abs P$ positions. Thus, each of them has an SLP representation 
 of size $\Oh(\abs P \cdot \log N)$. 
  \end{proof}

\begin{corollary}\label{cor:genalp}
Let $F$ be a free product of abelian groups and $\GG$ be a set of generators of $F$. 
Assume that the length of minimal solutions of equations over $F$ with alphabetic constraints can be bounded by some exponential function in $2^{n^{\Oh(1)}}$.
Then the question whether a given 
 Boolean formula of equations  over $F$ with  alphabetic constraints has a solution 
in $F$ can be decided in NP, i.e., in non-deterministic polynomial time. 
  \end{corollary}

\begin{proof}
 The NP-algorithm guesses an SLP for some solution of minimal length. 
 The size of the SLP has polynomial size. After that a  deterministic polynomial-time algorithm checks that the SLP is indeed a solution in reduced words
 verifying the alphabetic constraints. 
  \end{proof}
%%%%%%%%%%%%%%%%%%%%

\section{Hyperbolic groups}\label{sec:hyp}
In this section $G$ denotes a torsion-free non-elementary $\delta$-hyperbolic group which is generated by some finite subset $\Sig \sse G \sm \os 1$. 
As usual, we let $\GG = \Sig \cup \ov \Sig$ where 
$\ov \Sig = \Sig^{-1}$. We view $\GG$ as a finite alphabet with involution and we denote by $\pi: \GG^* \to G$ the canonical morphism onto $G$. 
For a word $w \in \GG^*$ we denote by $\abs w$ its length and 
by ${\abs w}_G$ its \emph{geodesic length}. It is the length of a shortest word 
$u$ such that $\pi(w) = \pi(u)$. Phrased differently,
${\abs w}_G$ is the length of a shortest path from $1$ to $\pi(w)$ in the Cayley graph  $\Cay(G,\Sig)$ of $G$ with respect to the generating set $\Sig$. As usual, a word $w \in \GG^*$ is called  \emph{geodesic},
if ${\abs w} = {\abs w}_G$. 
We say that a word $w \in \GG^*$ is \emph{$(\lam,d)$-quasi-geodesic},
if every factor $u$ of $w$ satisfies 
$$\abs u \leq \lam {\abs u}_G +d.$$ 
Note that a $(\lam,d)$-quasi-geodesic of length greater than $d$ can never represent the identity in $G$. 
A word $w \in \GG^*$ is called \emph{$\mu$-locally $(\lam,d)$-quasi-geodesic},
if every factor $u$ of $w$ which has length at most $\mu$ is $(\lam,d)$-quasi-geodesic. 
A fundamental property of a hyperbolic group is that local quasi-geodesics
are global quasi-geodesics for the appropriate choice of parameters. 
More precisely, \cite[Thm.~1.4]{CDP91} and \cite[Rem.~7.2.B]{gro87} provide for all $\lam, d$ 
an effective  bound 
for $\mu$  which is  polynomial in  $\lam +\del$
such that every 
$\mu$-local $(\lam,d)$-quasi-geodesic word is $(\lam',d')$-quasi-geodesic
with $\mu > d'$. Now, being $\mu$-locally $(\lam,d)$-quasi-geodesic is a local property which is therefore a ``rational constraint''. 
This fact has also been used in \cite{Dah} in order to show that the existential theory of a hyperbolic group is decidable. However, we need a more precise statement than being a  rational constraint. 
Let us have a closer look. 
\begin{lemma}\label{lem:syntac}
Let $u$ be a $\mu$-local $(\lam,d)$-quasi-geodesic word in $G$. Then there
is a word $v$ of length less than 
${\abs \GG}^{\mu}$ such that for all $x,y\in \GG^*$ the word 
$z = xuy$ is $\mu$-locally $(\lam,d)$-quasi-geodesic \IFF 
$z' = xvy$ is $\mu$-locally $(\lam,d)$-quasi-geodesic.
Moreover, if  $u \neq v$ then  
 $\abs v \geq  \mu-1$. 
\end{lemma}

\begin{proof}
Within this proof we abbreviate  ``$\mu$-locally $(\lam,d)$-quasi-geodesic''  by ``$\mu$-local''. The proof  follows from pigeon hole principle. 
%We define a congruence ${\sim}$ on $\GG^*$ as follows. First, we 
%let $z\sim z'$ if both words are not  $\mu$-local. %$(\lam,d)$-quasi-geodesic. 
%We may assume that such a word $z$  exists and we denote the 
%congruence class of $z$ by $0$ (\emph{zero}).
%Second, we let  $z\sim z'$ if $z = z'$ or if 
%$z,z'$ are (long)  $\mu$-local words such that we find factorizations 
%$ z = prq$ and $z'= psq$ for some $p,r,s,q$ with  $\abs p = \abs q = \mu -1$. We denote such a
%class of $z = prq$ by $[pq]$, but beware, it does not mean that $pq$ itself is
%$\mu$-local. %  $(\lam,d)$-quasi-geodesic words
%If the class $0 \neq [pq]$ exists, it means only that for some $r \in \GG^*$ the word $prq$ is  $\mu$-local. %  $(\lam,d)$-quasi-geodesic
%Since for all $x,y, z, z'\in \GG^*$ the relation  $z\sim z'$ 
%implies $xzy\sim xz'y$, the equivalence ${\sim}$ is indeed a congruence. 
%Hence, the classes form a finite monoid $M$ of size less than ${\abs \GG}^{2\mu}$. 
% 
% Now, let's turn to  the assertion in the lemma. 
We may assume 
 $\abs u \geq {\abs \GG}^{\mu}$ because otherwise we may choose $u = v$. 
 The word $u$ is longer than $\mu -1 + {\abs \GG}^{\mu-1}$
 because $\abs \GG \geq 2$. Hence there is factor $r$ of $u$ which 
 has length $\mu -1$ and which occurs at least twice. Therefore, we find 
 factorizations $u = prs = trq$ such that $p$ is a proper prefix of $t$. 
 Now, the word $v= prq$  is $\mu$-local and it shares the same prefix (suffix resp.) 
 of length $\mu -1$ as $u$. Thus, for all $x,y\in \GG^*$ the word 
$z = xuy$ is $\mu$-local % $(\lam,d)$-quasi-geodesic 
\IFF 
$z' = xvy$ is $\mu$-local. %  $(\lam,d)$-quasi-geodesic.
The word $v$ is shorter than $u$, but the length is at least 
$\abs r = \mu -1$.  We continue the process until we
end up in a word of length less than  ${\abs \GG}^{\mu}$. 
 \end{proof}
 
 In \cite{rs95} Rips and Sela have shown that solvability  of 
 equations in hyperbolic groups is decidable. Their techniques rely on the 
 notion of \emph{canonical representative}. This is a representation of an element of $G$ as an element over the free group $F(\Sig)$ (in $\GG^*$ resp.) satisfying  some 
 ``invariants''. In particular, 
 if $\theta(g)$ is a canonical representative of $g \in G$ then 
$g = \pi(\theta(g))$. We do not need the explicit definition of 
 a  canonical representative, but we need some crucial  properties.  The following result can be deduced from \cite{rs95} in a very similar way as done by Dahmani in  \cite[Prop.~3.4]{Dah} for relatively hyperbolic groups. Since the constants are different (and as they rely on  the PhD thesis 
 \cite{Dah_thesis}) we give a proof which refers to \cite{rs95}, only. 
 A statement as in \prref{lem:canrep} is not needed in \cite{rs95} since 
 the authors study systems of equations without inequalities, only. 
 %%%%%%%%%%%%%%%%%%%%%%%%%%%%%%
\begin{lemma}\label{lem:canrep}
Canonical representatives  (in the sense of \cite{rs95}) of elements of $G$  are $(\lam,d)$-quasi-geodesics for some
 $\lam, d \in \abs{\Sig}^{\Oh(\del)}$.
\end{lemma}
%%%%%%%%%%%%%%%%%%%%%%%%%%%%%%
\newcommand{\diw}[1]{\mathop{\mathrm{diff}_w}(#1)}
\begin{proof}
We follow the notation in \cite{rs95}. For vertices $x,$ $y$ in the 
 Cayley graph $\Cay(G, \Sig)$  we let $d(x,y)$ be its geodesic distance 
 and  $\abs x = d(x,1)$. We let $w\in \GG^*$ be some ``canonical representative`` of $\pi(w)$ in the sense of \cite{rs95}. Moreover, 
 let $K$ be the $2\del$ neighborhood of the path in $\Cay(G, \Sig)$ defined by some geodesic $\gam$ connecting $1$ and  $\pi(w)$ in $\Cay(G, \Sig)$. 
 Let $u\in \GG^*$ be a factor of  $w$. Hence $w = puq$. Define 
 vertices $\pi(p)$ and $\pi(pu)$ in $\Cay(G, \Sig)$. Then there are vertices $x,y \in K$ and so-called ``slices'' $S(x)$ and $S(y)$ with centers $x$ and  $y$ such that 
 $d(\pi(p),x)\leq 10 \del$ and $d(y, \pi(pu))\leq 10 \del$. 
 By definition of canonical representatives we have $\abs u \leq 20 \del n + 20 \del$, 
 where $n = \abs{\diw {x,y}}$,
  because then the number of ``slices'' between $x$ and $y$ is at most $n$. 
   Here $\diw {x,y}$ is the ``difference function'' applied to $(x,y)$. It remains to show that 
 $n \leq \lam d(x,y)+d $ with  $\lam, d \in {\abs{\Sig}}^{\Oh(\del)}$.
 According to \cite[Def.~3.3]{rs95}
 the number $\diw {x,y}$ is the difference between two non-negative numbers where  
 each of these numbers is the addition of two non-negative terms. 
 Moreover, there is some so-called ``cylinder'' $C$ such that, 
 by symmetries in $x$ and $y$ and in ``left'' and ``right'', we my assume
 $n /2 \leq L(y)\sm L(x)$ where 
 \begin{equation}\label{eq:hugo}
L(z') = \set{z\in C}{\abs {z} \leq \abs{ z'} \wedge d(z,z') \geq 10 \del}.
\end{equation}
 By \cite[Lem.~3.2]{rs95} we have $C \sse K$. Hence, by \prref{eq:hugo}
 \begin{equation}\label{eq:hogo}
n /2 \leq  \abs{\set{z\in K}{\abs x - 10 \del < \abs {z} \leq \abs y}}.
\end{equation}
 Indeed, if $z \in C$ with $\abs x - 10 \del \geq  \abs {z}$ then 
 $d(z,x) \geq 10 \del$ and $z\in L(x)$. Clearly, $\abs {z} > \abs y$ 
 implies $z \notin L(y)$ for all $z$. This shows (\ref{eq:hogo}).
 Since $x,y\in K$ there are $x',y' \in \gam$ such that  $d(x,x') \leq 2 \del$ and $d(y,y') \leq 2 \del$.
 In particular,  $d(x',y') \leq d(x,y) + 4 \del$.
 Moreover, 
  \begin{equation*}\label{eq:hagar}
\set{z\in K}{\abs x - 10 \del < \abs {z} \leq \abs y}\sse
\set{z\in K}{\abs {x'} - 12 \del < \abs {z} \leq \abs{y'} +2\del }.
\end{equation*}
Now, let $z \in K$ and $z'\in \gam$ such that $d(z,z') \leq 2 \del$ and 
$\abs {z'} \leq \abs {x'}- 14\del$ or $\abs {z'} > \abs {y'}+ 4\del$ then 
$\abs {z} \leq \abs {x}- 10\del$ or $\abs {z} > \abs {y}$. 
We conclude that for all $z \in L(y)\sm L(x)$ there is some 
$z' \in \gam$ with $\abs {x'}- 14\del \leq \abs {z'} \leq \abs {y'}+ 4\del$ such that 
$d(z,z') \leq 2 \del$. 
This implies  
\begin{equation}\label{eq:helga}
n /2 \leq  {\abs \GG}^{2\del}(d(x',y') + 18\del) \leq {\abs \GG}^{2\del}d(x,y) + {\abs \GG}^{2\del}22\del.
\end{equation}
 The result follows. 
 \end{proof}

 \begin{theorem}\label{thm:eqsyshyp}
There exists a polynomial $p(n)$ such that the following assertion holds:
Let $\cS$ be a system of  equations over a $\delta$-hyperbolic group generated by $\Sig$ and let $\sig$ 
be a  solution of length $N$.
Then there exists another solution $\sig'$ of length in ${\abs\Sig}^{\Oh(\del)} N$ and some SLP 
of  size $p({\abs\Sig}^{\del^2\log \del}+ \Abs{\cS} + {\log N})$ such that  $\sig'(X) = \eval(X)$ for all 
variables  used by $\sig'$.
\end{theorem}
 
\begin{proof} 
By standard arguments we may assume that $\cS$ is given 
by $n$ triangular equations $\cS$ of type $XYZ = 1$ and constraints $X = a$ where $X,Y,Z \in \OO$ and $a \in \GG$.  
The solution $\sig$ is given by some mapping $\sig: \OO_+ \to \GG^*$
and we may assume that $\sig(X)$ is geodesic for all $X \in \OO$ because this cannot increase the length $N$. 
Now, \cite{rs95} yields 
 an effective constant $\kap$ depending on $\del$ and $\abs \GG$ 
 and the following transformation of $\cS$. 
 \begin{itemize}
\item  With the help of fresh variables, each equation $XYZ = 1$ of $\cS$ is replaced by three equations 
  $$x=PA\overline{Q}, \quad 
  y= QB\ov R, \quad z=  R C \ov P.$$ 
  \item A constraint $X = a$ is replaced by the constraint
  $X = \theta(a)$ where $\theta(a)$ is some canonical representative of the letter $a$. 
\item The following  conditions are added:
  \begin{itemize}
\item ``$ABC = 1$ in $G$ and 
$\max\os{\abs{A}, \abs{B}, \abs{C}} \leq \kap n$''.
\end{itemize}
\end{itemize}
\cite{rs95} shows that it is possible to  choose canonical representatives $\theta(x)$ for 
all $x \in \sig(\OO) \cup \GG$ such that $\rho_\sig(X) = \theta(\sig(X))$
defines a solution ${\rho_\sig}: \OO_+ \to \GG^*$
for the new system over the free group $F(\Sig)$. 
Moreover, if $\rho'$ is any solution which respects the additional conditions and which
solves the new system over the free group $F(\Sig)$ then $\rho'$ solves $\cS$ over $G$, too. 
By \prref{lem:canrep} we know that the length of the solution 
 ${\rho_\sig}$ can be bounded by ${\abs\Sig}^{\Oh(\del)}N$ which is the first assertion in the theorem. The new system has a size which can be bounded by $\kap n \Abs\cS \leq \kap {\Abs\cS}^2$.
 The next step is to replace ${\rho_\sig}$ by some minimal solution 
 $\sig'$ for the system over the free group $F(\Sig)$ and hence for the original system $\cS$. The switch to $\sig'$ 
 does not increase the length with respect to ${\rho_\sig}$, but it allows to use
 \prref{thm:genalp}. It yields  a polynomial $p$ and an SLP for $\sig'$ of size $p( \kap +  \Abs{\cS} + {\del}\log {\abs\Sig} +{\log N})$ such that  $\sig'(X) = \eval(X)$ for all 
variables  used by $\sig'$.
It remains to estimate 
 $\kap$ by some polynomial in ${\abs\Sig}^{\del^2\log \del}$. This is done in \prref{lem:kappa}. 
 \end{proof}

\begin{lemma}\label{lem:kappa} 
The constant $\kappa$ in the proof of %\prref{prop:RS2} 
\prref{thm:eqsyshyp} can be estimated by  
$\kap \in {\abs \GG}^{\Oh(\del^{2}\log \del)}.$
\end{lemma}

\begin{proof}
The constant $\kap$  appears in \cite{rs95} as a product
of a function $f(\del) \in {\abs \GG}^{\Oh(\del)}$ times $\Chann(\mu_0)$.
Here $\Chann(\mu_0)$ is  an upper bound on the number of geodesics  in a $2\delta$-neighborhood
of a geodesic of length $\mu_0$ where  $\mu_0 \in \Oh(\del^{2}\log \del)$ by \cite[Def.~3.1]{rs95}.
Note that a geodesic contributing to $\Chann(\mu_0)$ can have length at most
$4 \del + \mu_0$. 
The size of such a neighborhood $U$ is therefore at most
$ \mu_0{\abs \GG}^{2\del}$.
In \cite{rs95} a doublly exponential bound for $\kap$ is used because  \cite{rs95} simply counts the number of all subsets of $U$. This number is  greater  than $2^{2^{\del}}$. However, a more accurate counting is possible.
Let us fix a starting point of a geodesic of length at most $4 \del + \mu_0$.
Then the geodesic can be described by a word in $\GG^*$  of length $4 \del + \mu_0$ or its prefix of length $4 \del + \mu_0-1$.
The number of words of length $\mu_0$ is ${\abs \GG}^{4 \del + \mu_0}$. This gives us the bound
$\Chann(\mu_0) \leq 2 \mu_0 {\abs \GG}^{2\del} {\abs \GG}^{4 \del + \mu_0}$.
Hence,
$\kap = f(\del)\cdot \Chann(\mu_0) \in {\abs \GG}^{\Oh(\del^{2}\log \del)}.$
\end{proof}

In \prref{thm:eqsyshyp} we did not treat inequalities because at present 
we have a worse estimation w.r.t.~compression by SLPs. We obtain 
a parameter which is unfortunately double-exponential in $\del$.
The can prove the following result. 
%%%%%%%%%%%%%%%%
\begin{corollary}\label{cor:boolhyp}
There exists a polynomial $p(n)$ such that the following assertion holds:
Let $\Phi$ be a Boolean formula of  equations over a $\delta$-hyperbolic group generated by $\Sig$ and let $\sig$ 
be a  solution of length $N$.
Then there exists another solution $\sig'$ of length in ${\abs\Sig}^{\Oh(\del)} N$ and some SLP 
of  size $p({2}^{2^{\Oh(\del\log \Sig)}}+ \Abs{\Phi} + {\log N})$ such that  $\sig'(X) = \eval(X)$ for all 
variables  used by $\sig'$.
\end{corollary}
%%%%%%%%%%%%%

\begin{proof}
The proof is almost identical to the proof of \prref{thm:eqsyshyp}. 
The additional difficulty is that we cannot replace the 
solution in canonical representatives by another solution over the free group. The problem is that $\rho(X) \neq 1$ in $F(\Sig)$ does not transfer 
to $\rho(X) \neq 1$ in $G$. We know however by construction that 
 $\rho_\sig(X) \neq 1$ in $G$ as soon as $\sig(X) \neq 1$ in $G$.
We also know to construct the generic solution $\wt {\rho_\sig}$ as explained in \prref{sec:gen}.  This solution has an SLP compression 
of polynomial size in ${\abs\Sig}^{\del^2\log \del}+ \Abs \Phi + {\log N}$ by \prref{thm:compall} and (the proof of) \prref{thm:eqsyshyp}. Our intention is to compress ${\rho_\sig}$;
and for that we must consider maximal free intervals. 
 The canonical representations ${\rho_\sig}(X)$ are $(\lam,d)$-quasi-geodesic for some $\lam, d \in {\abs \Sig}^{\Oh(\del)}$ by
\prref{lem:canrep}. 
Assume that we have ${\rho_\sig}(X) = puq$ where $u$ corresponds to some 
maximal free interval in $\wt {\rho_\sig}$ and $\sig(X) \neq 1$ in $G$. We know ${\rho_\sig}(X)\neq 1$ in $G$. But the problem is that $u$ might be long and incompressible. 
For some $\mu \in {\abs \Sig}^{\Oh(\del)}$ every $\mu$-local $(\lam,d)$-quasi-geodesic is in fact a $(\lam',d')$-geodesic with $\mu > d'+1$. Hence, if 
we choose $v$ such that $pvq$ is $\mu$-locally $(\lam',d')$-quasi-geodesic
and $\abs v > d'$ then $pvq \neq 1$ in the group $G$.
We care only if $\abs u >{\abs \GG}^{\mu}$. In this case we use 
\prref{lem:syntac} and we let $v\in \GG^*$ with 
$\mu -1 \leq \abs v < {\abs \GG}^{\mu}$ such that for all $x,y$ we have that $xuy$ is $\mu$-locally $(\lam,d)$-quasi-geodesic  \IFF $xvy$ is $\mu$-locally $(\lam,d)$-quasi-geodesic. This allows to substitute the maximal free interval belonging to $u$ by the 
word $v$. It might be that $v$ is not compressible, but at least we have a length bound on $v$. 
Iterating this process we obtain a new solution $\sig'$ satisfying the following conditions for all $X \in \OO$. 
\begin{itemize}
\item Every factor in $\sig'(X)$ which belongs to some maximal free interval has length less than $\abs{\GG}^\mu$.
\item The word $\sig'(X)$ is $\mu$-locally $(\lam,d)$-quasi-geodesic.
\item If $\sig'(X) \neq \sig(X)$
then $\abs{\sig'(X)} > d'$. In particular, $\sig'(X) \neq 1 \neq \sig(X)$ in the group $G$. 
\end{itemize}
The SLP for $\sig'$ can be constructed from the SLP for the generic solution 
$\wt {\rho_\sig}$ and writing all substitutions for maximal free intervals as plain words of length less that ${\abs \GG}^{\mu}$. We have ${\abs \GG}^{\mu} \in {2}^{{\Sig}^{\Oh(\del)}}$. Hence the result. 
\end{proof}

Dahmani  has shown that the existential theories of equations for hyperbolic groups are decidable, see \cite{Dah}. He does not mention explicit complexity bounds. Therefore, we add the 
following result.

\begin{proposition}\label{prop:pspacehyp}
Let $G$ be a finitely generated torsion-free $\delta$-hyperbolic group.
Then  the existential theory of equations over $G$ is in PSPACE.
\end{proposition}

\begin{proof}
 Since $G$ is fixed all parameters in ${2}^{{\Sig}^{\Oh(\del)}}$ become constants. 
By \cite{rs95} and the methods used in the proofs of \prref{cor:boolhyp}
we obtain an NP-reduction of the existential theory of equations over $G$
to the existential theory of equations with rational constraints in a fixed free 
finitely generated free group $F(\Sig)$. The later theory is in PSPACE by \cite{dgh05IC}.
\end{proof}

\begin{remark}\label{rem:conject}
We believe that  \prref{prop:pspacehyp} holds for also for hyperbolic groups with torsion. But we did not check enough details in \cite{DahmaniGui10} in order to make this statement rigorous. The reduction in the proof of \prref{prop:pspacehyp}
to the existential theory of  equations with rational constraints in $F(\Sig)$ creates
only rational constraints which involve finite monoids of polynomial size of the input. This is due to the local character to test the constraint of being 
 $\mu$-locally $(\lam',d')$-quasi-geodesic where (for a fixed group $G$) the values $\mu, \lam, d$ are 
 constants. As also supported by \prref{cor:boolhyp} we  conjecture 
 that the  existential theory of equations in a fixed finitely generated (torsion-free) $\delta$-hyperbolic group $G$ is in NP. 
As soon as $G$ contains a non-abelian free subgroup the problem is known to be NP-hard (even for systems of quadratic equations) by a recent result in \cite{KMTV2013}.
\end{remark}

%%%%%%%%%%%%%%%%%%%%%%%%% 
\section{Toral relatively hyperbolic groups}\label{sec:toral}
In this section we will obtain results similar to the results of the previous section for systems of equations in toral relatively hyperbolic groups using the work of Dahmani \cite{Dah}.
We will use the following definition of relative hyperbolicity. A f.g. group $G$ with  generating set $\Sigma$ is relatively hyperbolic relative to a collection of finitely generated subgroups $\mathcal P=\{P_1,\ldots ,P_k\}$ if the Cayley graph $\Cay (G,\Sigma\cup \Pi)$  (where  $\Pi$ is the set of all non-trivial elements of subgroups in $\mathcal P$) is a hyperbolic metric space, and the pair $\{G,\mathcal P\}$ has \emph{Bounded Coset Penetration} property (BCP property for short). 
 The pair $(G,\{P_1,P_2,...,P_k\})$ satisfies the \emph{BCP property}, if   for any $\lambda ≥ 1,$
there exists constant $a = a(\lambda)$ such that the following conditions hold. Let $p, q$ be $(\lambda, 0)$-quasi-geodesics without backtracking in $\Cay(G, \Sigma\cup \Pi)$  such that their initial points coincide ($p_- = q_-$), and for the terminal points $p_+,q_+$ we have  $\dist_\Sig(p_+,q_+) \leq 1.$
 
1) Suppose that for some $ i$, $s$ is a $P_i$-component of $p$ such that $\dist_\Sig(s_-,s_+) \geq a;$ then there exists a $P_i$-component $t$ of $q$ such that $t$ is connected to $s$ (there exists a path $c$ in $\Cay(G, \Sigma\cup \Pi)$ that connects some vertex of $p$ to some vertex of $q$ and the label of this path is a word consisting of letters from $P_i$).
 
2) Suppose that for some $i,$ $s$ and $t$ are connected $P_i$-components of $p$ and $q$ respectively. Then $\dist_{\Sigma}(s_-,t_-) \leq a$ and $\dist_{\Sigma}(s_+,t_+) \leq a.$ 
 
A group $G$ that is hyperbolic relative to a collection $\{P_{1},\ldots,P_{k}\}$ of subgroups is called \emph{toral}, if $P_{1},\ldots,P_{k}$ are all abelian and $G$ is
torsion-free. In this section we always assume that $\Sigma$ contains generators of all subgroups  $P_{1},\ldots,P_{k}$.

In \cite{Dah} Dahmani has shown that the satisfiability of systems of equations and inequalities is decidable in toral relatively hyperbolic groups. He also uses the notion of canonical representatives, the canonical representatives in this case are elements of the free product $\tilde G=F(\Sigma)\ast P_1\ast\ldots\ast P_k.$
In \cite{Dah}, Section 2.4.2  the language $\mathcal L$ of so called {\em geometric}
 elements in $\tilde G$ is introduced. These are elements  $\tilde\gamma\in\tilde G$ that do not have any $\theta$-detour such that
$\pi(\tilde\gamma)$ in a $L$-local $(L_1,L_2)$-quasi-geodesic in $Cay (G, \Sigma\cup\Pi).$
The constants are defined in \cite{Dah}, Section 2.4.2, as  $L_1=10^4\delta M, L_2=10^6\delta ^2M,$ where $M$ is a bound on the cardinality of cones of radius and angle $50\delta$. 
\begin{lemma} \label{lem:7.1} The cardinality of a cone of radius and angle $\ell$ is bounded by 
$C(\ell)^{\ell}$, where $C(\ell)$ is the number of circuits in $Cay(G,\Sigma\cup\Pi) $ of length less than $\ell.$  Moreover  $C\leq |\Gamma|^{6(a(\ell)+1)a(\ell)},$ where $a(\ell)$ is the BCP constant for the group. 
\end{lemma}
\begin{proof} The proof of the first statement repeats the proof of \cite{Dah1}, Corollary 1.7. 

Now we have to estimate the constant $C(\ell)$ in terms of $a(\ell)$. This can be done using \cite{Dah_thesis}, Proposition 1 in the Appendix. This proposition shows that each circuit of length $\ell$ in $\Cay (G, \Sigma\cup \Pi)$ is formed by two $\ell$-quasi-geodesics both belonging to a fixed  ball of radius $\ell(a(\ell)+1).$  Therefore, the number of such circuits is bounded by $|\Gamma|^{6(a(\ell)+1)\ell}.$ 
\end{proof}

By this lemma,  $C(50\delta)\leq |\Gamma|^{6(a(50\delta)+1)a(50\delta)}.$ And $M\leq |\Gamma|^{300\delta(a(50\delta)+1)a(50\delta)}.$
The angle $\theta$ can be taken as $10^4(D+60\delta),$ where $D$ is a fellow traveling constant for $1000\delta$-quasi-geodesics, greater that any angles at finite valency vertices. Therefore \cite{Dah}, Proposition 3.4 implies
\begin{lemma}  
Canonical representatives  (in the sense of \cite{Dah}) of elements of $G$  are $(\lam,d)$-quasi-geodesics for some
 $\lam, d \in {\abs\Sig}^{\Oh(\del (a(50\delta))^2)}$.
 
\end{lemma}

\begin{theorem}\label{thm:rel}
There exist  polynomials $p(n), q(n)$ such that the following assertion holds. 
Let $\cS$ be a system of  equations over a toral relatively hyperbolic group with hyperbolicity constant $\delta$ for $Cay(G,\Sigma\cup\Pi)$ and BCP function $a(\ell),$ generated by $\Sig$. Let $\sig$ 
be a  solution of length $N$.
Then there exists another solution $\sig'$ of length in ${\abs\Sig}^{\Oh(\del a(50\del))} N$ and some SLP 
of  size $p({\abs\Sig}^{q(\del a(\delta ^3))}+ \Abs{\cS} + {\log N})$ such that  $\sig'(X) = \eval(X)$ for all 
variables  used by $\sig'$.
\end{theorem}

\begin{proof} By standard arguments we may assume that $\cS$ is given 
by $n$ triangular equations $\cS$ of type $XYZ = 1$ and constraints $X = a$ where $X,Y,Z \in \OO$ and $a \in \GG$.  
The solution $\sig$ is given by some mapping $\sig: \OO_+ \to \GG^*$
and we may assume that $\sig(X)$ is geodesic for all $X \in \OO$ because this cannot increase the length $N$. 
Now, \cite{Dah} yields 
 an effective constant $\kap$ depending on $\del$ and $\abs \GG$ 
 and the following transformation of $\cS$. 
 \begin{itemize}
\item  With the help of fresh variables, each equation $XYZ = 1$ of $\cS$ is replaced by three equations 
  $$x=PA\overline{Q}, \quad 
  y= QB\ov R, \quad z=  R C \ov P.$$ 
  \item A constraint $X = a$ is replaced by the constraint
  $X = \theta(a)$ where $\theta(a)$ is some canonical representative of the letter $a$. 
\item The following  conditions are added:
  \begin{itemize}
\item ``$ABC = 1$ in $G$ and 
$\max\os{\abs{A}, \abs{B}, \abs{C}} \leq \kap n$''.
\end{itemize}
\end{itemize}

Let us show that  $\kappa$ is  exponential in $\delta$ and $|\Sigma|$.
To estimate $\kappa$ we have to estimate the function $\phi (n)$ in \cite[Thm.~2.22 ]{Dah1}, because $\kappa n$ has the order of $\phi (n)$, size of the holes in the slice decomposition of a cylinder, see \cite[Sec.~2.4]{Dah1}. The function $\phi (n)$ is defined in \cite[Sec.~2.3]{Dah1},  as well as all necessary constants, $$\phi (n)=24(n+1) \Capa (\mu) (2\epsilon +1)\epsilon ,$$ where $\epsilon=N_{1000\delta,\delta}$ that has the order of $\delta ^3$, $\mu=100 N_{1000\delta,\delta}+(1000\delta)^2$ also has the order of $\delta ^3$ and $\Capa (\mu)$ is the number of different channels of segments of length $\mu$. If $g=[v_1,v_2]$ is a segment of length $\mu$ then we have to estimate the number of  geodesics not shorter than $|v_2  -v_1|$ that stay
in the union of the cones of radius and angle $\epsilon$ centered in the edges of $g$. By  \prref{lem:7.1},  the cardinality of such a cone is bounded by $C(\epsilon)^{\epsilon}.$
The number of geodesics in one such cone is bounded by 
$C(\epsilon)^{\epsilon}$ times the bound on the number of paths of length $\leq 2\epsilon$.
The number of paths is bounded by $m^{2\epsilon}=2^{(\log m)2\epsilon}$.
Therefore, the number  of channels of a segment of length $\mu$ is bounded by $C(\epsilon)^{\epsilon\mu}2^{(\log m)2\epsilon\mu}.$

Finally $\Capa(\mu)\leq 2^{6(\log m)\epsilon\mu(a(\epsilon)+1)^2}.$ This gives the desired estimate for $\phi (n)$ and $\kappa$.

\end{proof}
%\bibliographystyle{abbrv}
%\bibliography{traces}
%\bibliography{references,../TRACES/traces}{}

\begin{thebibliography}{10}

\bibitem{AlstrupBR00}
S.~Alstrup, G.~S. Brodal, and T.~Rauhe.
\newblock Pattern matching in dynamic texts.
\newblock In D.~B. Shmoys, editor, {\em SODA}, pages 819--828. ACM/SIAM, 2000.

\bibitem{CDP91}
M.~Coornaert, T.~Delzant, and A.~Papadopoulos.
\newblock {\em G{\'e}om{\'e}trie et th{\'e}orie des groupes. {L}es groupes
  hyperboliques de M.~{G}romov}, volume 1441 of {\em Lecture Notes in
  Mathematics}.
\newblock Springer, 1991.

\bibitem{Dah_thesis}
F.~Dahmani.
\newblock {\em Les groupes relativement hyperboliques et leurs bords}.
\newblock PhD thesis, Universit{\'e} Louis Pasteur, Strasbourg, 2003.

\bibitem{Dah1}
F.~Dahmani.
\newblock Accidental parabolics and relatively hyperbolic groups.
\newblock {\em Israel Journal of Mathematics}, 153:93–127, 2006.

\bibitem{Dah}
F.~Dahmani.
\newblock Existential questions in (relatively) hyperbolic groups.
\newblock {\em Israel Journal of Mathematics}, 173:91--124, 2009.

\bibitem{DahmaniGui10}
F.~Dahmani and V.~Guirardel.
\newblock Foliations for solving equations in groups: free, virtually free and
  hyperbolic groups.
\newblock {\em J. of Topology}, 3:343--404, 2010.

\bibitem{dgh05IC}
V.~Diekert, C.~Guti{\'{e}}rrez, and {\Ch}.~Hagenah.
\newblock The existential theory of equations with rational constraints in free
  groups is {PSPACE}-complete.
\newblock {\em Information and Computation}, 202:105--140, 2005.
\newblock Conference version in STACS 2001, LNCS 2010, 170--182, 2004.

\bibitem{GasieniecKPR96}
L.~Gasieniec, M.~Karpinski, W.~Plandowski, and W.~Rytter.
\newblock Efficient algorithms for {L}empel-{Z}ip encoding ({E}xtended
  abstract).
\newblock In R.~G. Karlsson and A.~Lingas, editors, {\em SWAT}, volume 1097 of
  {\em Lecture Notes in Computer Science}, pages 392--403. Springer, 1996.

\bibitem{gro87}
M.~Gromov.
\newblock Hyperbolic groups.
\newblock In S.~M. Gersten, editor, {\em Essays in Group Theory}, number~8 in
  MSRI Publ., pages 75--263. Springer-Verlag, 1987.

\bibitem{hagenahdiss2000}
{\Ch}.~Hagenah.
\newblock {\em Gleichungen mit regul{\"a}ren {R}andbedingungen {\"u}ber freien
  {G}ruppen}.
\newblock Ph.d.-thesis, Institut f{\"u}r {I}nformatik, {U}niversit{\"a}t
  {S}tuttgart, 2000.

\bibitem{Jez12icalp}
A.~Jez.
\newblock Faster fully compressed pattern matching by recompression.
\newblock In A.~Czumaj, K.~Mehlhorn, A.~M. Pitts, and R.~Wattenhofer, editors,
  {\em ICALP (1)}, volume 7391 of {\em Lecture Notes in Computer Science},
  pages 533--544. Springer, 2012.

\bibitem{Jez13dlt}
A.~Jez.
\newblock Recompression: {W}ord equations and beyond.
\newblock In M.-P. B{\'e}al and O.~Carton, editors, {\em Developments in
  Language Theory}, volume 7907 of {\em Lecture Notes in Computer Science},
  pages 12--26. Springer, 2013.

\bibitem{KharlampovichLMT10}
O.~Kharlampovich, I.~Lys{\"e}nok, A.~Myasnikov, and N.~Touikan.
\newblock The solvability problem for quadratic equations over free groups is
  {NP}-complete.
\newblock {\em Theory of Computing Systems}, 47:250--258, 2010.

\bibitem{KMTV2013}
O.~{Kharlampovich}, A.~{Mohajeri}, A.~{Taam}, and A.~{Vdovina}.
\newblock {Quadratic Equations in Hyperbolic Groups are {NP}-complete}.
\newblock {\em ArXiv e-prints}, 2013.

\bibitem{koz77}
D.~Kozen.
\newblock Lower bounds for natural proof systems.
\newblock In {\em Proc.~of the 18th Ann.~Symp. on Foundations of Computer
  Science, FOCS'77}, pages 254--266, Providence, Rhode Island, 1977. IEEE
  Computer Society Press.

\bibitem{lohrey06siam}
M.~Lohrey.
\newblock Word problems and membership problems on compressed words.
\newblock {\em SIAM J. Comput.}, 35:1210--1240, 2006.

\bibitem{Lohrey2012survey}
M.~Lohrey.
\newblock Algorithmics on {SLP}-compressed strings: {A} survey.
\newblock {\em Groups Complexity Cryptology}, 4:241--299, 2012.

\bibitem{LohreyS07}
M.~Lohrey and S.~Schleimer.
\newblock Efficient computation in groups via compression.
\newblock In V.~Diekert, M.~V. Volkov, and A.~Voronkov, editors, {\em CSR},
  volume 4649 of {\em Lecture Notes in Computer Science}, pages 249--258.
  Springer, 2007.

\bibitem{lot02}
M.~Lothaire.
\newblock {\em Algebraic Combinatorics on Words}, volume~90 of {\em
  Encyclopedia of Mathematics and its Applications}.
\newblock Cambridge University Press, Cambridge, 2002.

\bibitem{MehlhornSU97}
K.~Mehlhorn, R.~Sundar, and C.~Uhrig.
\newblock Maintaining dynamic sequences under equality tests in polylogarithmic
  time.
\newblock {\em Algorithmica}, 17(2):183--198, 1997.

\bibitem{pla94}
W.~Plandowski.
\newblock Testing equivalence of morphisms on context-free languages.
\newblock In J.~van Leeuwen, editor, {\em Proc. Algorithms---{ESA}'94}, volume
  855 of {\em Lecture Notes in Computer Science}, pages 460--470, Utrecht, The
  Netherlands, 1994. Springer.

\bibitem{pr98icalp}
W.~Plandowski and W.~Rytter.
\newblock Application of {L}empel-{Z}iv encodings to the solution of word
  equations.
\newblock In K.~G. Larsen et~al., editors, {\em Proc. 25th International
  Colloquium Automata, Languages and Programming (ICALP'98), Aalborg (Denmark),
  1998}, number 1443 in Lecture Notes in Computer Science, pages 731--742,
  Heidelberg, 1998. Springer-Verlag.

\bibitem{rs95}
E.~Rips and Z.~Sela.
\newblock Canonical representatives and equations in hyperbolic groups.
\newblock {\em Inventiones Mathematicae}, 120:489--512, 1995.

\bibitem{schleimer08}
S.~Schleimer.
\newblock Polynomial-time word problems.
\newblock {\em Commentarii Mathematici Helvetici}, 83:741--765, 2008.

\bibitem{sch91}
K.~U. Schulz.
\newblock {M}akanin's algorithm for word equations --- {T}wo improvements and a
  generalization.
\newblock In K.~U. Schulz, editor, {\em Word Equations and Related Topics},
  number 572 in Lecture Notes in Computer Science, pages 85--150, Heidelberg,
  1991. Springer-Verlag.

\end{thebibliography}
\newcommand{\Ju}{Ju}\newcommand{\Ph}{Ph}\newcommand{\Th}{Th}\newcommand{\Ch}{Ch}\newcommand{\Yu}{Yu}\newcommand{\Zh}{Zh}\newcommand{\St}{St}

\end{document}